\documentclass[11pt,reqno]{amsart}
\usepackage{graphicx}
\pdfoutput=1
\usepackage{amsfonts,amsmath,amssymb}
\usepackage{hyperref}
\usepackage{paralist}
\usepackage[linesnumbered,ruled,vlined,norelsize]{algorithm2e}

\usepackage{thmtools}
\declaretheoremstyle[bodyfont=\normalfont]{noncursive}
\declaretheorem{theorem}
\declaretheorem[numberwithin=section]{lemma}

\declaretheorem[numberlike=lemma]{proposition}

\declaretheorem[style=noncursive,numberlike=lemma]{definition}

\declaretheorem[style=noncursive,numberlike=lemma]{remark}

\newcommand{\im}{\ensuremath{\mbox{\rm Im}\,}}

\newcommand{\CC}[1]{\mathbb{C}^{#1}}
\newcommand{\CP}[1]{\mathbb{CP}^{#1}}

\numberwithin{equation}{section}

\def\label#1{\label{#1}{\sf (#1)}~}

\def\endpf{\hbox{\vrule height1.5ex width.5em}}

\def\<{\langle}
\def\>{\rangle}
\numberwithin{equation}{section}
\def\<{\langle}
\def\>{\rangle}

\def\-{\overline}

\def\endpf{\hbox{\vrule height1.5ex width.5em}}

\def\endpf{\hbox{\vrule height1.5ex width.5em}}

\def\-{\overline}

\def\endpf{\hbox{\vrule height1.5ex width.5em}}

\def\endpf{\hbox{\vrule height1.5ex width.5em}}

\def\-{\overline}

\begin{document}

\title{\bf On the embeddability of  real hypersurfaces into hyperquadrics }
\author{Ilya Kossovskiy}
\address{\parbox{0.8\linewidth}{%
               Department of Mathematics and Statistics, Masaryk University, Brno, Czechia/\\
               Faculty of Mathematics, University of Vienna, Austria}
    }
\email{kossovskiyi@math.muni.cz, ilya.kossovskiy@univie.ac.at}
\author{Ming Xiao}
\address{Department of Mathematics, University of Illinois at Urbana-Champaign}
\email{mingxiao@illinois.edu}

\begin{abstract}
A well known result of Forstneri\'c \cite{For} states that most real-analytic strictly pseudoconvex hypersurfaces in complex space are not holomorphically embeddable into spheres of higher dimension. A more recent result by Forstneri\'c \cite{For1} states even more: most real-analytic hypersurfaces do not admit  a holomorphic embedding even into a merely algebraic hypersurface of higher dimension, in particular, a hyperquadric. Explicit examples of real-analytic hypersurfaces non-embaddable into hyperquadrics were obtained by Zaitsev \cite{Z}. In contrast, the classical theorem of Webster \cite{webster} asserts that {every} {\em real-algebraic} Levi-nondegenerate hypersurface admits a transverse holomorphic embedding into a nondegenerate real hyperquadric in complex space.

In this paper, we provide {\em effective} results on the non-embeddability of real-analytic hypersurfaces into a hyperquadric. We show that, for any $N >n \geq 1$, the defining functions $\varphi(z,\bar z,u)$ of all  real-analytic hypersurfaces $M=\{v=\varphi(z,\bar z,u)\}\subset\CC{n+1}$ containing Levi-nondegenerate points and locally transversally holomorphically embeddable into some hyperquadric $\mathcal Q\subset\CC{N+1}$ satisfy an {\em universal} algebraic partial differential equation $D(\varphi)=0$, where the algebraic-differential operator $D=D(n,N)$ depends on $n, N$ only. To the best of our knowledge, this is the first effective result characterizing  real-analytic hypersurfaces  embeddable into a hyperquadric of higher dimension.
As an application, we show that  for every $n,N$ as above there exists $\mu=\mu(n,N)$ such that a Zariski generic real-analytic hypersurface $M\subset\CC{n+1}$ of degree $\geq \mu$ is not transversally holomorphically embeddable into any hyperquadric $\mathcal Q\subset\CC{N+1}$. We also provide an explicit upper bound for $\mu$ in terms of $n,N$. To the best of our knowledge, this
gives the first effective lower bound for the CR-complexity of a Zariski generic real-algebraic hypersurface in complex space of a fixed degree.
\end{abstract}

\maketitle

\tableofcontents

\section{Introduction}
Let $M\subset\CC{n+1},\,n\geq 1$ be a real-analytic Levi-nondegenerate hypersurface. The celebrated theory due to Chern and Moser \cite{CM} (see also Cartan \cite{cartan}) asserts that only a very rare such $M$ admits a local biholomorphic mapping into a nondegenerate real hyperquadric
$$\mathcal Q=\bigl\{[\xi_0,...,\xi_{n+1}]\in\CP{n+1}:\,\, |\xi_0|^2+...+|\xi_k|^2-|\xi_{k+1}|^2-...-|\xi_{n+1}|^2=0\bigr\}.$$
Moreover, Chern and Moser show that the existence of the desired biholomorphic mapping into a hyperquadric is equivalent to vanishing of a special {\em CR-curvature} of a real hypersurface $M$.

A natural problem to pursue in view of the Chern-Moser theory is the possibility to construct a local holomorphic embedding $F:\,(M,p)\mapsto (\mathcal Q,p')$ of a real-analytic hypersurface $M\subset\CC{n+1},\,n\geq 1$ into a hyperquadric $\mathcal Q\subset\CC{N+1}$ {\em of higher dimension}. Here by a  holomorphic  embedding $F$ of $M\subset {\mathbb C}^n$ into $M'
\subset {\mathbb C}^N$, we mean a holomorphic embedding of an open
neighborhood $U$ of $M$ in $\CC{n}$ into a neighborhood $U'$ of $M'$ in $\CC{N}$, sending
$M$ into $M'$. One usually presumes certain nondegeneracy conditions for the mapping $F$, such as {\em transversality} (the latter means that $dF(\CC{n+1})|_p\not\subset T_{p'} \mathcal Q$).

The existence of a transversal holomorphic embedding into a hyperquadric can be viewed as a {\em finite CR-complexity } of a real hypersurface (see, e.g., Ebenfelt and Shroff \cite{es}). The latter number is the minimal possible difference $N-n$ between the CR-dimensions of the target hyperquadric and the source real hypersurface. An alternative approach to complexity in CR-geometry is due to the school of D'Angelo, see, e.g., \cite{Da,Da1,DL2}.  A strong motivation for studying the embedding problem is the celebrated theorem of Webster \cite{webster} which states that {\em every} real-algebraic Levi-nondegenerate hypersurface admits a transverse holomorphic embedding into a nondegenerate real hyperquadric in complex space. Thus, every algebraic Levi-nondegenerate hypersurface has a finite CR-complexity.

Since the work of Webster, a large number of publications have been dedicated to studying holomorphic embeddings of real hypersurfaces into hyperquadrics. However, despite of the extensive research in this direction, the following problem remains widely open:


\smallskip

\noindent{\bf  Problem 1:}~{ Characterize the embeddability of a real hypersurface $M \subset \mathbb{C}^{n+1}$ into a hyperquadric $\mathcal{Q}^{2N+1} \subset \mathbb{C}^{N+1}.$ More precisely, find a necessary and sufficient condition for $M$ to admit a  transversal holomorphic embedding into some $\mathcal{Q}^{2N+1} \subset \mathbb{C}^{N+1}.$}

\smallskip

We emphasize in connection with Problem 1 that not every Levi-nondegenerate real-analytic hypersurface can be transversally holomorphically embedded into a hyperquadric. Indeed, a well known result of Forstneri\'c \cite{For} (see also Faran \cite{Fa}) states that most real-analytic strictly pseudoconvex hypersurfaces are not holomorphically embeddable into spheres of higher dimension. A more recent result by Forstneri\'c \cite{For1} states even more: most real-analytic hypersurfaces do not admit  a holomorphic embedding even into a merely algebraic hypersurface of higher dimension. Importantly, both cited theorems are proved by showing that the set of embeddable hypersurfaces is  a set of  first Baire category. An important step towards understanding the embeddability property was done by Zaitsev \cite{Z}, who obtained explicit examples of Levi-nondegenerate real-analytic hypersurfaces  that are not transversally holomorphically embeddable into any hyperquadrics. We also mention  the recent work of Huang and Zaitsev \cite{HZ} and Huang and Zhang \cite{HZh},  where the authors construct concrete algebraic Levi-nondegenerate
hypersurfaces with positive signature which can not be holomorphically embedded into a hyperquadric with the same signature
of any dimension.

 However, the cited results still leave open the question on an {\em effective} characterization of the set of real-analytic Levi-nondegenerate hypersurfaces in $\CC{n+1}$, admitting a local transversal holomorphic embedding into a hyperquadric in $\CC{N+1}$. That is, we are searching for a more constructive characterization of the set of embeddable hypersurfaces than the one in \cite{For1}. \autoref{T1} below provides such a characterization for {\em any} fixed $n,N$. Namely, we show that for any fixed $n,N$ with $n\geq 1,\,n< N$  the set of embeddable hypersurfaces $M=\{v=\varphi(z,\bar z,u)\}\subset\CC{n+1}$ satisfies an {\em universal} algebraic partial differential equation $$D(\varphi)=0,$$ where the differential-algebraic operator $D=D(n,N)$ depends on $n,N$ only. Thus, the defining functions of embeddable hypersurfaces form a subset of a {\em differential-algebraic set}.
Following the method of the present paper, each differential-algebraic operator $D(n,N)$ can be {\em effectively} computed (see \autoref{explicit} below), as well as an effective bound for its degree can be obtained immediately (see Appendix I).

The other question addressed in the paper is connected to Webster's embedding theorem mentioned above. Motivated by embedding theorems in various geometries (such as Whitney  embedding theorem in differential topology and Remmert theorem in the Stein space theory) it is natural, in view of Webster's theorem, to ask the following.

\smallskip

\noindent{\bf Problem 2.}~{\em Is there a uniform embedding dimension $N$ which only depends on $n$
such that all Levi-nondegenerate real-algebraic hypersurfaces $M \subset \mathbb{C}^{n+1}$ can be transversally holomorphically
embedded into a hyperquadric of suitable signature in $\mathbb{C}^{N+1}$ ? In other words, is there a uniform upper bound for the CR-complexity of all Levi-nondegenerate real-algebraic hypersurfaces $M\subset\CC{n+1}$? }

\smallskip\rm

A closely related problem is as follows. 

\smallskip

\noindent{\bf Problem 3.}~{\em Provide an effective bound for the CR-complexity of a (generic) real-algebraic hypersurface $M\subset\CC{n+1}$ of a fixed degree $k$ in terms of $n$ and $k$.}

\smallskip\rm

By applying \autoref{T1},  we give a negative answer to Problem 2 (see \autoref{T2}). Moreover, \autoref{T2} gives an {\em explicit} constant $\mu=\mu(n,N)$ such that a Zariski-generic algebraic hypersurface of any fixed degree $k\geq \mu$ is not transversally holomorphically embeddable into any hyperquadric in $\CC{N+1}$, thus providing a solution for Problem 3.



We now formulate our results in detail. We first recall the concept of a  differential-algebraic operator. Let $n,l\geq 1$ be integers and $P$ be a polynomial defined on the space $J^l(\CC{n},\CC{})$ of jets of maps from $\CC{n}$ to $\CC{}$. Then $P$ uniquely defines a {\em differential-algebraic operator} $\mathcal D=\mathcal D(P)$, which is the differential operator acting on an analytic function $\rho:\,U\mapsto\CC{}$ by
$$\mathcal D(\rho):= P(j^l\rho)$$
 (here $U\subset\CC{n}$ is a domain). The integer $l$ is called its {\em order}. We  call a differential-algebraic operator {\em shift-invariant}, if it is invariant under shifts in $j^0\rho$ (that is, $\mathcal D(\rho)$ does not depend on $z_1,...,z_n,\rho$ explicitly and depends on  derivatives of $\rho$ of order at least $1$).

\begin{theorem} \label{T1}
For any integers $N> n\geq 1$, there exists a universal non-zero shift-invariant differential-algebraic operator $D=D(n,N)$
such that the following holds. If
a real-analytic hypersurface $M \subset \mathbb{C}^{n+1}$ with a defining equation
$$v=\varphi(z, \overline{z},u)$$
contains at least one Levi-nondegenerate point and admits a local transverse holomorphic embedding into a hyperquadric $\mathcal{Q}^{2N+1}\subset\mathbb{C}^{N+1}$ near some point $p_{0} \in M$, then

$$D(\varphi)\equiv 0.$$
\end{theorem}




To the best of our knowledge, \autoref{T1} gives first effective results characterizing  real-analytic hypersurfaces  embeddable into a hyperquadric of higher dimension. We shall also note that a weaker version of the assertion of \autoref{T1} (a differential-algebraic relation for jets of the Segre varieties $\{Q_p\}_{p\in Q_0}$) was proved earlier by Zaitsev in  \cite{Z}.

We now give a series of important remarks.


\begin{remark}\label{2n} We remark that, according to \cite{ehz2}, we may drop the transversality requirement in \autoref{T1} in the case $N<2n$, as the latter holds automatically.
\end{remark}

\begin{remark}\label{explicitbound} In Section 3 we show that in the case $n=1,N=2$ the order of the differential-algebraic operator in \autoref{T1} equals to $18$. An explicit bound for the order of  $D(n,N)$ in the general case  can be verified from the Appendix I.
\end{remark}

\begin{remark}\label{explicit}
The proof of \autoref{T1} given in Sections 3 and 4 in fact provides an algorithm for finding the   differential-algabraic operator $D(\varphi)$ {\em explicitly} for any fixed dimensions $n,N$. More precisely, the proof shows that  $D(\varphi)$ is a finite product of explicit determinants involving the {\em PDE defining function} $\{\Phi_{ij}\}_{i,j=1}^n$ of a real hypersurface and its derivatives (see Section 2.2 for details of the concept). In Sections 3 and 4 we also provide an algorithm for recalculating all the relevant derivatives of the PDE defining function $\{\Phi_{ij}\}_{i,j=1}^n$ in terms of derivatives of the initial defining function $\varphi(z,\bar z,u)$.
\end{remark}

By using \autoref{T1}, we obtain the following effective bound for embeddability of real-algebraic hypersurfaces into hyperquadrics.

\begin{theorem} \label{T2}
For any integers $N> n\geq 1$,  there exists  $\mu=\mu(n,N)$ such that a Zariski generic  real-algebraic hypersurface  $M \subset \mathbb{C}^{n+1}$ of any degree $k\geq \mu$  is not transversally holomorphically embeddable  into a hyperquadric $\mathcal{Q}^{2N+1}\subset\mathbb{C}^{N+1}$. An explicit bound for $\mu(n,N)$ is given in \autoref{thebound} below (see Appendix I).
\end{theorem}

Thus, \autoref{T2} provides a solution for Problem 2. As  was mentioned above, it also gives the first effective lower bound for the CR-complexity of a Zariski-generic real-algebraic hypersurface in complex space of a fixed degree, thus giving a solution to Problem 3.
\begin{remark}
Note that in the special case $n=1,\,N=2$ we may take $\mu = 18$ for the degree bound, as shown in Section 3.
\end{remark}


The main tool of the paper is the recent dynamical technique in CR-geometry, which shall be addressed as the {\em method of associated differential equations} in the non-singular setting (see Sukhov \cite{sukhov1,sukhov2}), and the {\em CR -- DS technique} in the singular one (see the work of Lamel, Shafikov and the first author \cite{divergence,nonminimalODE,nonanalytic}). An overview of the technique is given in Section 2.2.

\bigskip

\begin{center} \bf Acknowledgements. \end{center}

\smallskip

The authors acknowledge AIM for holding a workshop on Cauchy-Riemann equations in several variables in June 2014, where they started the work. The authors are grateful to Valeri Beloshapka, John D'Angelo, Peter Ebenfelt, and Xiaojun Huang for their interest and helpful discussions, and are particularly grateful to Dmitri Zaitsev for his very valuable comments on the initial version of the manuscript, which allowed to  significantly strengthen the main results of the paper. 

The first author is supported by the Austrian Science Fund (FWF).

\section{Preliminaries}

\subsection{Transversality of CR-embeddings}

\mbox{}

\bigskip

We first recall the notion of transversality. If $U$ is an open subset of $\mathbb{C}^{n+1}, H$ a holomorphic mapping $U \mapsto \mathbb{C}^{N+1},$ and $M'$ a real hypersurface through a point $H(p)$ for some $p \in U,$
then $H$ is said to be transversal to $M'$ at $H(p)$ if
$$T_{H(p)}M' + dH(T_{p}\mathbb{C}^{n+1})=T_{H(p)}\mathbb{C}^{N+1}, $$
where $T_{p}\mathbb{C}^{n+1}$ and $T_{H(p)}M'$ denote the real tangent spaces of $\mathbb{C}^{n+1}$
and $M'$ at $p$ and $H(p),$ respectively. We here mention that in our setting where  there is a real hypersurface $M \subset U$ such that $H(M) \subset M',$  the notion of transversality of a mapping to a hypersurface coincides with that of CR transversality(cf. \cite{ber}). We also recall that $H$ is called CR transversal if $dF(\mathbb CT_pM)$ is not contained in $\mathcal V'_{F(p)}+\overline{\mathcal V'_{F(p)}},$
where $\mathcal V'$ is the CR bundle of $M'.$ Note that a CR mapping is CR transversal at $p \in M$ is equivalent to the nonvanishing of the derivative of its normal component at $p$ along the normal direction(cf. \cite{ber}).

\bigskip

\subsection{Description of the principal method} It was observed by Cartan \cite{cartan,cartanODE} and Segre \cite{segre} (see also Webster \cite{webster}) that the geometry of a real hypersurface in $\CC{2}$ parallels that of a second order ODE
\begin{equation}\label{wzz}
w\rq{}\rq{}=\Phi(z,w,w\rq{}).
\end{equation}
More generally,  the geometry of a real hypersurface in $\CC{n+1},\,n\geq 2$ parallels that of a complete second order system of PDEs
\begin{equation}\label{wzkzl}
w_{z_kz_l}=\Phi_{kl}(z_1,...,z_n,w,w_{z_1},...,w_{z_n}),\quad k,l=1,...,n.
\end{equation}
  Moreover, this parallel becomes algorithmic by using  the Segre family of a real hypersurface. With any real-analytic Levi-nondegenerate hypersurface $M\subset\CC{n+1},\,n\geq 1$ one can uniquely associate an ODE \eqref{wzz} ($n=1$) or a PDE system \eqref{wzkzl} ($n\geq 2$). The Segre family of $M$ plays a role of a mediator between the hypersurface and the associated differential equations.  A modern clear exposition of the method was given in the work \cite{sukhov1,sukhov2} of Sukhov.

  The associated differential equations procedure is particularly clear in the case of a Levi-nondegenerate hypersurface in $\CC{2}$. In this case the Segre family is a 2-parameter anti-holomorphic family of pairwise transverse holomorphic curves. It immediately follows then from the main ODE theorem that there exists a unique ODE \eqref{wzz}, for which the Segre varieties are precisely the graphs of solutions. This ODE is called \it the associated ODE. \rm

Let us provide some details in the general case. We denote the coordinates in $\CC{n+1}$ by $(z,w)=(z_1,...,z_n,w)$. Let $M\subset\CC{n+1}$ be a smooth real-analytic
hypersurface, passing through the origin, and choose a small neighborhood $U$
 of the origin. In this case
we associate a complete second order system of holomorphic PDEs to $M$,
which is uniquely determined by the condition that the differential equations are satisfied by all the
graphing functions $h(z,\zeta) = w(z)$ of the
Segre family $\{Q_\zeta\}_{\zeta\in U}$ of $M$ in a
neighbourhood of the origin.
To be more explicit we consider the
so-called {\em  complex defining
 equation } (see, e.g., \cite{ber})\,
$w=\rho(z,\bar z,\bar w)$ \, of $M$ near the origin, which one
obtains by substituting $u=\frac{1}{2}(w+\bar
w),\,v=\frac{1}{2i}(w-\bar w)$ into the real defining equation and
applying the holomorphic implicit function theorem.
 The Segre
variety $Q_p$ of a point $p=(a,b)\in U,\,a\in\CC{n},\,b\in\CC{}$ is  now given
as the graph
\begin{equation} \label{segredf}w (z)=\rho(z,\bar a,\bar b). \end{equation}
Differentiating \eqref{segredf} once with respect to all variables, we obtain
\begin{equation}\label{segreder} w_{z_j}=\rho_{z_j}(z,\bar a,\bar b),\,j=1,...n. \end{equation}
Considering \eqref{segredf} and \eqref{segreder}  as a holomorphic
system of equations with the unknowns $\bar a,\bar b$, an
application of the implicit function theorem yields holomorphic functions
 $A_1,...,A_n, B$ such that
$$
\bar a_j=A_j(z,w,w'),\,\bar b=B(z,w,w').
$$
The implicit function theorem applies here because the
Jacobian of the system coincides with the Levi determinant of $M$
for $(z,w)\in M$ (\cite{ber}). Differentiating \eqref{segredf} twice
and substituting for $\bar a,\bar b$ finally
yields
\begin{equation}\label{segreder2}
w_{z_kz_l}=\rho_{z_kz_l}(z,A(z,w,w'),B(z,w,w'))=:\Phi_{kl}(z_1,...,z_n,w,w_{z_1},...,w_{z_n}),\,k,l=1,...,n.
\end{equation}
Now \eqref{segreder2} is the desired complete system of holomorphic second order PDEs $\mathcal E = \mathcal{E}(M)$.
 \begin{definition}\label{PDEdef}
 We call the PDE system $\mathcal E = \mathcal{E}(M)$  \it the system of PDEs, associated with $M$. \rm   We also call the collection   $\{\Phi_{ij}\}_{i,j=1}^n$ {\em the PDE defining function} of a Levi-nondegenerate hypersurface $M$.
\end{definition}

For further developments of the associated differential equations method see, e.g., \cite{nonminimalODE},\cite{divergence},\cite{nonanalytic},\cite{analytic} and references therein.

\section{Real-analytic hypersurfaces in $\CC{2}$ embeddable into hyperquadrics in $\CC{3}$}
In this section we  prove \autoref{T1} in the more transparent case when $n=1$ and $N=2$,  so that the source $M\subset\CC{2}$ and the target quadric $\mathcal Q\subset\CC{3}$. In what follows $(z,w)=(x+iy,u+iv)$ denote the coordinates in $\CC{2}$ and  $(Z_1,Z_2,W)$ denote that in $\CC{3}$. We start with the observation that, due to the polynomial nature of differential-algebraic operators, {\em it is sufficient to prove the existence of the desired differential-algebraic operator in a neighborhood of an arbitrary point $p_0\in M$ when $M$ is given by the same defining equation $v=\varphi(z, \overline{z},u)$} (at all other points the identity $D(\varphi\equiv 0$ is satisfied then by analyticity).
We first write a holomorphic embedding map
$$F=(f_1, f_2,  g):\,(M,p_{0})\mapsto (\mathcal Q, F(p_{0}))$$
for some $p_{0} \in M.$ Assume $F$ is holomorphic in a small neigborhood $U$ of $p_0$ in $\mathbb{C}^2.$ Shifting the base point $p_0$, we may assume $M$ to be Levi-nondegenerate at $p_0$. We split our arguments into two cases: 

{\bf Case I:}~The image of $U$ under $F$ is contained in some affine linear subspace of $\mathbb{C}^3$ and thus maps $M$ into a hyperquadric in $\CC{2}$ (in which case $M$ is biholomorphic to the sphere $S^3\subset\CC{2}$);

{\bf Case II:}~The image of $U$ under $F$ is not contained in any affine linear subspace of $\mathbb{C}^3.$

\bigskip

The case I of a spherical hypersurface $M$ is considered later separately, so that we assume now to be under the setting of Case II.  

By changing the base point and shifting the coordinates, we may assume that $p_0\in M$ is the origin, and $M$ is Levi-nondegenerate at $0$. 
Let us write the target quadric $\mathcal Q$ in the form
$$\im W=Z_1\overline{Z_1}\pm Z_2\overline{Z_2}.$$
After a change of coordinates in  $\mathbb{C}^{3}$ preserving the quadric we may assume $F(0)=0$.


Let us then consider the Segre family $\{S_p\}_{p\in U}$ of $M$ ($U\subset\CC{2}$ is a neighborhood of the origin). In view of $F(M)\subset\mathcal Q$, any Segre variety $S_{p}$ of  a point $p=(a,b)\in U$, considered as a graph $w= w(z)= \rho(z, \bar a, \bar b)$, is contained
in the Segre variety of $F(p)=(A,B,C)$. Thus
we have,
\begin{equation}\label{segre}
\frac{g-\overline{C}}{2i}=f_1\overline{A}\pm f_2\overline{B}|_{w=w(z)}.
\end{equation}
We now differentiate \eqref{segre} three times with respect to $z$ and write the result in terms of the $3$-jet
$$(z,w,w',w'',w''')$$
of a Segre variety $S_p$ at a point $(z,w)\in S_p$. Note that each differentiation amounts to applying the vector field
$$\mathcal{L}:=\frac{\partial}{\partial z} + w' \frac{\partial}{\partial w}$$
with the rule $\frac{\partial}{\partial z}w^{(j)}=w^{(j+1)},\,\frac{\partial}{\partial w}w^{(j)}=0$. Performing the differentiation $3$ times, we get:
\begin{equation}\label{segre1}
 (\mathcal{L}f_1)\overline{A}\pm (\mathcal{L}f_2)\overline{B}-\frac{1}{2i}(\mathcal{L} g)=0.
\end{equation}

\begin{equation}\label{segre2}
 (\mathcal{L}^2f_1)\overline{A}\pm (\mathcal{L}^2f_2)\overline{B}-\frac{1}{2i}(\mathcal{L}^2 g)=0.
\end{equation}

\begin{equation}\label{segre3}
 (\mathcal{L}^3f_1)\overline{A}\pm (\mathcal{L}^3f_2)\overline{B}-\frac{1}{2i}(\mathcal{L}^3 g)=0.
\end{equation}

Considering \eqref{segre1}-\eqref{segre3} as a linear system for the unknowns $(\overline{A}, \pm\overline{B},  -\frac{1}{2i}),$
we  conclude that it has a non-zero solution, thus its determinant is $0.$ That is,

\begin{equation}\label{detzero}
\mathrm{det}\left(
\begin{array}{llll}
\mathcal{L}f_1 &   \mathcal{L}f_{2} & \mathcal{L}g \\
\mathcal{L}^2f_1 &  \mathcal{L}^2f_{2} & \mathcal{L}^2g \\
\mathcal{L}^3f_1 &  \mathcal{L}^3f_{2} & \mathcal{L}^3g\\
\end{array}
\right)\equiv 0.
\end{equation}

Note that
$$\mathcal{L}h = \frac{\partial h}{\partial z}+ w'\frac{\partial h}{\partial w};$$

$$\mathcal{L}^2 h= \frac{\partial^2 h}{\partial z^2}+ 2w' \frac{\partial^2 h}{\partial z \partial w} +
(w')^2\frac{\partial^2 h}{\partial w^2} + w'' \frac{\partial h}{\partial w};$$

$$\mathcal{L}^3 h= \frac{\partial^3 h}{\partial z^3} + 3w'' \frac{\partial^2 h}{\partial z \partial w} + 3w' \frac{\partial^3 h}{\partial z^2 \partial w}
+ 3w''w' \frac{\partial^2 h}{\partial w^2} + 3(w')^2 \frac{\partial^3 h}{\partial z\partial w^2} + (w')^3 \frac{\partial^3 h}{\partial w^3} +
 w'''\frac{\partial h}{\partial w}.$$


Hence \eqref{detzero} reads as

\begin{equation}\label{3order}
w'''P(z,w, w') +R(z,w, w',w'')=0,
\end{equation}
where $P,R$ are polynomials of the form
\begin{equation}\label{PQ}
\begin{aligned}
P=&\chi_{0}+ \chi_{1}w' + \chi_{2} (w')^2,\quad
R= \Bigl(\chi_3+\chi_4w'+\chi_5(w')^2+\chi_6(w')^3+\chi_7(w')^4+\\
+&\chi_8(w')^5
+\chi_9(w')^6\Bigr)
+ \Bigl(\chi_{10}+\chi_{11}w'+\chi_{12}(w')^2+\chi_{13}(w')^3\Bigr)w''+\Bigl(\chi_{14}+\chi_{15}w'\Bigr)(w'')^2.
\end{aligned}
\end{equation}
Here all $\chi_j=\chi_j(z,w)$. We prove the following lemma on $\chi_j.$

\begin{lemma}\label{lemmachi}
The functions $\chi_j, 0 \leq j \leq 15,$ do not all vanish identically.
\end{lemma}

\begin{proof}
 Assume, otherwise, that  all $\chi_j, 0 \leq j \leq 15,$ are identical zeroes.  Then the left hand sides of (\ref{detzero}) and (\ref{3order}) are identically zero when $w', w'', w'''$ are regarded as independent variables:
\begin{equation}\label{eqnprt}
P(z,w, \eta_1)=0, \quad Q(z,w, \eta_1, \eta_2)=0
\end{equation}
for any $\eta_1, \eta_2 \in \mathbb{C}.$ 

We claim that \eqref{eqnprt} implies the following: for any (fixed) polynomial complex curve $$\Gamma=\bigl\{(z,w)\in\CC{2}:\,\,  w=h(z)\bigr\}$$  passing through the origin, the components of the map $F|_\Gamma$ are linearly dependent. Indeed, let us write $$\Lambda:=\frac{\partial}{\partial z}+h'(z) \frac{\partial}{\partial w}$$ as the holomorphic tangent vector of $\Gamma.$ (Note that applying $\Lambda$ to a holomorphic function amounts to differentiating it along $\Gamma$).  By the above observation (\ref{eqnprt}), we conclude that
\begin{equation}
\mathrm{det}\left(
\begin{array}{llll}
\Lambda f_1 &   \Lambda f_{2} & \Lambda g \\
\Lambda^2 f_1 &  \Lambda^2 f_{2} & \Lambda^2 g \\
\Lambda^3 f_1 &  \Lambda^3 f_{2} & \Lambda^3 g\\
\end{array}
\right)=h'''P(z,w, h') +R(z,w, h',h'')=0~\text{on}~\Gamma.
\end{equation}
By the classical property of Wronskian, we conclude that $$\lambda_1 \Lambda f_1+ \lambda_2 \Lambda f_2 +\lambda_3 \Lambda g|_\Gamma\equiv 0$$  for some complex numbers $\lambda_i, 1 \leq i \leq 3,$ that are not all zero.  Furthermore, the assumption $F(0)=0$ yields $\lambda_1 f_1+ \lambda_2 f_2 +\lambda_3 g=0$ on $\Gamma.$ Thus  the components of the map $F|_\Gamma$ are linearly dependent, and this proves the claim.

By the assumption of Case II, there exist three distinct points $p_i=(a_i, b_i), 1 \leq i \leq 3,$ near $0$ such that
\begin{equation}\label{eqnpide}
\mathrm{Span}_{\mathbb{C}}\{ F(p_1), F(p_2), F(p_3)\}=\mathbb{C}^3.
\end{equation} Perturbing $p_i$ if necessary, we can assume $a_i \neq a_j$ for $i \neq j$ and $a_i \neq 0$ for each $i$. We then choose a holomorphic polynomial $h_0(z)$ such that $h_0(a_i)=b_i, 1 \leq i \leq 3,$ and $h_0(0)=0.$ Hence the origin and $p_i$
are all on the complex curve $\Gamma_0$ defined by $w=h_0(z).$ Now the assertion of the above the  claim applied for $\Gamma_0$ gives a contradiction to (\ref{eqnpide}). This establishes the lemma.
\end{proof}


On the other hand, $M$ is Levi-nondegenerate at $0$ and thus its  Segre family satisfies a second order ODE
\begin{equation}\label{ODE}
w''= \Phi(z,w,w')
\end{equation}
for a holomorphic near a point $(0,0,\xi_0)$ function $\Phi$.
We now consider  the $3$-jet space $J^3(\CC{},\CC{})$ with the coordinates $(z,w,\xi,\eta,\zeta)$ (where $\xi,\eta,\zeta$ correspond to $w',w'',w'''$ respectively) and treat the ODEs \eqref{3order},\eqref{ODE} as respectively submanifolds $\mathcal M,\mathcal E$  in $J^3(\CC{},\CC{})$.
Then $\mathcal M$ looks as
\begin{equation}\label{manifold1}
 P(z,w,\xi)\zeta+R(z,w,\xi,\eta)=0
\end{equation}
and $\mathcal E$ as
\begin{equation}\label{manifold2}
\eta=\Phi(z,w,\xi), \quad \zeta=\Phi_{z}+ \Phi_{w}\xi+ \Phi_{\xi}\eta.
\end{equation}
Now the fact that each graph of a  solution of \eqref{ODE} is contained in that of \eqref{3order} implies
$$\mathcal E\subset\mathcal M,$$
so that
\begin{equation}\label{condition}
 \Bigl(\Phi_{z}+ \Phi_{w}\xi+ \Phi_{\xi}\Phi\Bigr)P(z,w,\xi)+R(z,w,\xi,\Phi)=0
\end{equation}
(where $\Phi=\Phi(z,w,\xi)$).

Substituting \eqref{PQ} into \eqref{condition}, we obtain a scalar {\em linear} equation for the functions $\chi_0(z,w),...,\chi_{15}(z,w)$ with coefficients depending on $z,w,\xi$. Differentiating this equation $15$ time with respect to the variable $\xi$, we obtain $15$ new identities each of which is a similar scalar linear equation for the functions $\chi_0(z,w),...,\chi_{15}(z,w)$. In view of the fact that Lemma \ref{lemmachi}, i.e., not all $\chi_j$ vanish identically, so that the determinant
$$\mathcal D\Bigl(\Phi(z,w,\xi)\Bigr)$$
of the corresponding $16\times 16$ linear system vanishes identically. Note that this determinant is nothing but  {\it a universal differential-algebraic polynomial $\mathcal D$ of order $16$ applied to the function $\Phi$}. Note that $\mathcal D$ is invariant under shifts in $z,w$.
Thus, the embeddability of the hypersurface $M$ into $\mathcal Q$ implies
\begin{equation}\label{DPhi}
\mathcal D(\Phi)\equiv 0.
\end{equation}

 We finally consider Case I, in which $M$ is spherical. In the latter case, the Segre family of $M$ is locally biholomorphic to the family of straight lines in $\CC{2}$, and hence (see Tresse \cite{tresse}) $\Phi(z,w,\xi)$ is cubic in the argument $\xi$ (this could be proved by arguments similar to the ones above, but we will not provide here the proof of this classical fact). Now let us write, for the function $\Phi$ under consideration, the above determinant $\mathcal D(\Phi)$. In view of the fact that $\Phi$ is cubic in $\xi$, the derivatives of the first row of order $\geq 8$ vanish identically, so that we conclude that $\mathcal D(\Phi)\equiv 0$  in this special case as well.

We shall note that in the latter, equidimensional, case an differential-algebraic operator characterizing the sphericity of a Levi-nondegenerate hypersurface $M\subset\CC{2}$ was obtained in the work \cite{merker} of Merker.

We now need
\begin{proposition}\label{Dnotzero}
The differential-algebraic operator $\mathcal D$ is not identical zero.
\end{proposition}
\begin{proof}
We claim that there exists a function $\Phi$ of the form $\Phi=\Phi(\xi)$ such that $\mathcal D(\Phi)$ does not vanish identically. Indeed, substituting $\Phi=\Phi(\xi)$ into \eqref{condition} and differentiating $15$ times in $\xi$, we obtain a $16\times 16$ determinant which, in turn, is {\it the Wronskian of the system of functions in the first row}. This first row has the form
\begin{equation}\label{row}
\Bigl(\Phi_\xi\Phi,\xi\Phi_\xi\Phi,\xi^2\Phi_\xi\Phi,1,\xi,\xi^2,...,\xi^6, \Phi,\xi\Phi,\xi^2\Phi,\xi^3\Phi,
\Phi^2,\xi\Phi^2\Bigr).
\end{equation}
We then choose an analytic $\Phi(\xi)$ in such a way that the collection of functions in \eqref{row} is linearly independent (this is possible since every linear dependence between the components of \eqref{row} implies a nontrivial  algebraic or differential equation for $\Phi$). Then the Wronskian $\mathcal D(\Phi)$ does not vanish, and this proves the claim and the proposition.

\end{proof}

Write $w=\rho(z,\overline{z},\overline{w})$ as the complex defining function of $M.$ We now aim to express the condition $\mathcal D\Bigl(\Phi(z,w,\xi)\Bigr)=0$ in the form
$$D^C\Bigl(\rho(z,a,b)\Bigr)=0$$ for some  differential-algebraic operator $D^C$ of order $18$. Indeed, in view of the
Levi-nondegeneracy of $M$ near $0$ the map
$$(z,a,b)\mapsto \Bigl(z,\rho(z,a,b),\rho_{z}(z,a,b)\Bigr)$$
is a local biholomorphism between $\CC{3}\,(z,a,b)$ and $\CC{3}\,(z,w,\xi)$. If
$$(z,w,\xi)\mapsto \Bigl(z,A(z,w,\xi),B(z,w,\xi)\Bigr)$$ is the inverse biholomorphism,
then (compare with the associated differential equation procedure described in section 2) we have
$$w=\rho(z,a,b),\quad \xi=\rho_z(z,a,b),\quad \Phi(z,w,\xi)=\rho_{zz}(z,a,b),$$
where $a=A(z,w,\xi),\, b=B(z,w,\xi)$. Thus $\Phi$ is already expressed in terms of $\rho$. Let us then demonstrate, for example, how we express $\Phi_w(z,w,\xi)$. We have:
\begin{equation}\label{Phiw}
\Phi_w=\rho_{zza}A_w+\rho_{zzb}B_w.
\end{equation}
We now need to compute $A_w,B_w$ in terms of $\rho$. For that, we differentiate the identities
$$w=\rho(z,A(z,w,\xi), B(z,w,\xi)),\quad \xi=\rho_z(z,A(z,w,\xi), B(z,w,\xi))$$
in $w$ and get:
\begin{equation}\label{systemAB}
1=\rho_aA_w+\rho_bB_w,\quad 0=\rho_{az}A_w+\rho_{bz}B_w.
\end{equation}
Note that \eqref{systemAB} is a linear system for $A_w,B_w$ with the determinant at the reference point $(z,w,\xi)=(0,0,\xi_0)$ being equal to the Levi determinant of $M$ at the origin. Thus this determinant is non-zero and, applying the Cramer rule, we find $A_w,B_w$ as rational functions of the $2$-jet of $\rho$. Substituting into \eqref{Phiw}, we then find $\Phi_w$ as a rational function of the $3$-jet of $\rho$.

We then similarly express the entire $16$-jet of $\Phi$ as a rational (vector-valued) function of the $18$-jet of $\rho$. Thus the condition   $\mathcal D\Bigl(\Phi(z,w,\xi)\Bigr)\equiv 0$ reads as $D^C\Bigl(\rho(z,a,b)\Bigr)=0$ for some order $18$ differential-algebraic operator $D^C$, as required. Note that $D^{\CC{}}$ is invariant under shifts in $z,a,b$. Also note that the differential-algebraic operator $D^C$ is not identical zero. Indeed, by \autoref{Dnotzero}, there is a $\Phi(z,w,w')$ with $\mathcal D(\Phi)\not\equiv 0$, and for such $\Phi(z,w,\xi)$ we find $\rho(z,a,b)$ with $D^{\CC{}}(\rho)=\mathcal D(\Phi)\not\equiv 0$ by solving the ODE $w''=\Phi(z,w,w')$ with the initial data $w(0)=a, w'(0)=b.$

Now it is not difficult to complete the proof or \autoref{T1} in the case of CR-dimension $1$.

\begin{proof}[Proof of \autoref{T1} for $n=1,N=2$]
\mbox{}

Recall that the complex defining function $\rho(z,a,b)$ and the real defining function $\varphi(z,a,u)$ are connected  via the identities:
$$\rho(z,a,b)=u+i\varphi(z,a,u), \quad b=u-i\varphi(z,a,u).$$
Then, for example, the derivative $\rho_b$ is expressed via the $1$-jet of $\varphi$ as follows: we have
$$\rho_b=u_b+iu_b\varphi_u,\quad 1=u_b-iu_b\varphi_u,$$
so that
$$\rho_b=(1+i\varphi_u)/ (1-i\varphi_u).$$
We similarly expressed the entire $18$-jet of $\rho$ as a rational (vector-valued) function of the $18$-jet of $\varphi$. Thus the condition   $D^C\Bigl(\rho(z,w,\xi)\Bigr)\equiv 0$ reads as $D\Bigl(\varphi(z,a,b)\Bigr)=0$ for some order $18$ differential-algebraic operator $D$.

It remains to show that the operator $D$ is non-trivial (on the space of real-analytic defining functions $\varphi$ of real hypersurfaces). For that we note that $D$ is identical zero if and only if it is identical zero on the subspace of $\varphi$ defining a real hypersurface (since  the latter subspace is totally real). However, $D(\varphi)$ is not identical zero since $D^C(\rho)$ is not identical zero, as was shown above. This proves the theorem for $n=1,N=2$.

\end{proof}

\begin{proof}[Proof of \autoref{T2} for $n=1,N=2$]
Let us denote by $V_k$ the space of  polynomials $\varphi(z,a,u)$  of degree  $\leq k$ for some $k\geq 18$. We claim that there exists $\varphi\in V_k$ such that $D(\varphi)$ does not vanish identically. Indeed, the identity $D(\varphi)=0$ defines a proper algebraic variety $A$ in the jet bundle $J^{18}(\CC{3},\CC{})$ (the properness follows from the non-triviality of $D$). Picking a point $q\in J^{18}(\CC{3},\CC{})\setminus A$ we choose the unique polynomial $\psi\in V_k$ of degree $18$ with the $18$-jet corresponding to $q$, and get $D(\varphi)\not\equiv 0$, as required. Thus $D$ is generically non-vanishing on $V_k$.

If we now consider the set $W_k$ of $\varphi(z,a,u)$ arising from an algebraic equation $P(z,a,u,v)=0$ for a polynomial $P$ of degree $\leq k$, then $W_k$ has a structure of algebraic manifold. Hence either $D$ vanishes identically on $W_k$, or is (Zariski) generically non-vanishing. By the above argument, we conclude that $D$  is (Zariski) generically non-vanishing on $W_k$, and this implies the claim of the theorem for $n=1,N=2$.

\end{proof}

\section{The high dimensional case}
In this section  \autoref{T1} and \autoref{T2} will be established in the general case. 

For a fixed $n\geq 1$, we set $\mathcal{M}_{m},\, m \geq n,$ to be the set of all Levi-nondegenerate
hypersurfaces in $\mathbb{C}^{n+1}$ that can be locally tranversally holomorphically embedded into a hyperquadric $\mathcal{Q}^{2m+1} \subset \mathbb{C}^{m+1}.$ We also write $\tilde{\mathcal{M}}_{m}
\subset \mathcal{M}_{m}$ to be the collection of Levi-nondegenerate hypersurfaces $M$ in $\mathbb{C}^{n+1}$ satisfying the following property (*) (for a fixed $m$):

\bigskip

{\bf Property (*)}~{\em There exist a point $p \in M$ and a holomorphic map $F$ from a small neighborhood $U$ of $p$ to $\mathbb{C}^{m+1}$  such that $M$ is locally transversally holomorphically embedded into $\mathcal{Q}^{2m+1} \subset \mathbb{C}^{m+1}$ by $F.$ Moreover,
the image of $U$ under $F$ is not contained in any affine linear subspace of $\mathbb{C}^{m+1}.$}

\bigskip

Note that when $m=n,$ the assumption in Property (*) that the image of $U$  is not contained in any affine linear subspace of $\mathbb{C}^{m+1}$ can be dropped, as it follows already from the transversality assumption.  We obviously have

\begin{equation}\label{union}
\bigcup_{m=n}^N \tilde{\mathcal{M}}_{m}= \mathcal{M}_{N}.
\end{equation}

We  prove in this section the following two theorems implying \autoref{T1} and \autoref{T2} respectively.
\begin{theorem}\label{T4}
For any integers $m \geq n\geq 1,$  there exists an universal non-zero shift-invariant differential-algebraic operator $D=D(n,m)$ such that the following holds. If
a germ of real-analytic hypersurface $M \subset \mathbb{C}^{n+1}$ with a defining equation
$$v=\varphi(z, \overline{z},u)$$ is contained in $\tilde{\mathcal{M}}_{m},$ i.e., satisfies property (*) for $m,$ then
$$D(\varphi)\equiv 0.$$
\end{theorem}

\begin{theorem}\label{T5}
For any pair of integers $m \geq n \geq 1,$ there exists a positive integer $\nu=\nu(n,m)$  such that a Zariski generic  real-algebraic hypersurface  $M \subset \mathbb{C}^{n+1}$ of any degree $ \geq \nu$  is not contained in $\tilde{\mathcal{M}}_m.$
\end{theorem}

In what follows $(z,w)=(z_{1},...,z_{n},w)$ denote the coordinates in $\mathbb{C}^{n+1}$ and $(Z_{1},...,Z_{m},W)$ denote that in $\mathbb{C}^{m+1}.$ Write
the holomorphic embedding map

$$ F=(f_{1},...,f_{m}, g): (M,p_{0}) \rightarrow (\mathcal{Q}^{2m+1}, F(p_{0})),$$
for some Levi-nondegenerate point $p_{0}.$
As in Section 3, by shifting the base point of the coordinates, we can assume $p_{0}=0.$

Let us write the target hyperquadric $\mathcal{Q}^{2m+1}= \mathcal{Q}^{2m+1}_{l}$ in the form
$$\mathrm{Im} W=-Z_{1} \overline{Z}_{1}-...-Z_{l}\overline{Z}_{l}+ Z_{l+1}\overline{Z}_{l+1}+...+ Z_{m}\overline{Z}_{m},$$
where $l$ is the signature of $\mathcal{Q}^{2m+1}.$



Let us now consider the Segre family $\{ S_{p}\}$ of $M.$ In view of $F(M) \subset \mathcal{Q}^{2m+1},$ any Segre variety $S_{p}$ of a point $p=(a,b)=(a_{1},...,a_{n}, b) \in U,$
considered as a graph $w=w(z)=\rho(z,\overline{a}, \overline{b})$ is contained in the Segre variety of $F(p)=(A_{1},...,A_{m},C).$ Thus we have:
\begin{equation}\label{Gsegre}
\frac{g-\overline{C}}{2i}=-f_{1}\overline{A}_{1}-...-f_{l}\overline{A}_{l}+ f_{l+1}\overline{A}_{l+1}+...+f_{m}\overline{A}_{m}|_{w=w(z)}.
\end{equation}

As in Section 3, we differentiate several times with respect to $z$ and write the result in terms of the $\mu-$jet 

$$(z,w,w^{(\alpha)})_{1 \leq |\alpha| \leq \mu}$$
of a Segre variety $S_{p}$ at a point $(z_{1},...,z_{n},w) \in S_{p}$ for some $\mu.$ Here we use the notation
$$w^{(\alpha)}=\frac{\partial^{|\alpha|} w}{\partial z_{1}^{\alpha_{1}}...\partial z_{n}^{\alpha_{n}}}$$ for any multiindex $\alpha=(\alpha_{1},...,\alpha_{n}).$ For first order derivatives we use the notation $w'_{j}= \frac{\partial w}{\partial z_{j}}.$

For the following, we fix $$k:=m-n+1.$$ 
Write the basis of holomorphic tangent vectors along $S_{p}$ as
$$ \mathcal{L}_{j}=\frac{\partial}{\partial z_{j}}+ w'_{j}\frac{\partial }{\partial w}, 1 \leq j \leq n.$$
 Next, for a multi-index $\alpha=(\alpha_1,...,\alpha_n)$ we write $\mathcal{L}^{\alpha}=\mathcal{L}_{1}^{\alpha_{1}}...\mathcal{L}_{n}^{\alpha_{n}}.$
We apply $\mathcal{L}^{\alpha^{1}},...,\mathcal{L}^{\alpha^{m}},\mathcal{L}^{\alpha^{m+1}}$ to equation (\ref{Gsegre})
for some multi-indices $\alpha^1,\cdots, \alpha^{m+1}$ (precise form of which will  be determined later) with each $$|\alpha^i| \leq m-n+2=k+1.$$  
As the result of the differentiations, we obtain:
\begin{equation}\label{Gsegrealpha}
-\mathcal{L}^{\alpha^{j}}\overline{A}_{1}-...-\mathcal{L}^{\alpha^{j}}f_{l}\overline{A}_{l}+ \mathcal{L}^{\alpha^{j}}f_{l+1}\overline{A}_{l+1}+...
+\mathcal{L}^{\alpha^{j}}f_{m}\overline{A}_{m}-\frac{1}{2i}(\mathcal{L}^{\alpha^{j}}g)=0,\quad 1 \leq j \leq m+1.
\end{equation}

Equations \eqref{Gsegrealpha} form a linear system for $(m+1)$ unkowns $(-\overline{A}_{1},...,-\overline{A}_{l},\overline{A}_{l+1},...,\overline{A}_{m},-\frac{1}{2i}).$ Note that it has a non-zero
solution, thus its determinant is $0.$ That is,
\begin{equation}\label{Gdetzero}
 \mathrm{det} \left(  \begin{matrix}
 \mathcal{L}^{\alpha^{1}}f_{1} & ...  & \mathcal{L}^{\alpha^{1}} f_{m} & \mathcal{L}^{\alpha^{1}} g \\
 ... & ... & ...  & ...  \\
 \mathcal{L}^{\alpha^{m}} f_{1} & ...  & \mathcal{L}^{\alpha^{m}} f_{m} & \mathcal{L}^{\alpha^{m}} g \\
 \mathcal{L}^{\alpha^{m+1}}f_{1} & ... & \mathcal{L}^{\alpha^{m+1}} f_{m} &  \mathcal{L}^{\alpha^{m+1}} g
\end{matrix}\right) \equiv 0~\text{on}~ S_{p}.
\end{equation}

Note that for a function $h$ we have
$$\mathcal{L}_{j}h = \frac{\partial h}{\partial z_{j}}+ w'_{j}\frac{\partial h}{\partial w};$$
$$\mathcal{L}_{i}\mathcal{L}_{j}h=\frac{\partial^2 h}{\partial z_{i}\partial z_{j}}+w'_{i}\frac{\partial^2 h}{\partial z_{j} \partial w}+
w'_{j}\frac{\partial^2 h}{\partial z_{i} \partial w} +w'_{i}w'_{j}\frac{\partial^2 h}{\partial w^2}+ w''_{ij}\frac{\partial h}{\partial w}.$$
In general, for a multiindex $\alpha_j,~ \mathcal{L}^{\alpha^j}h$  is a polynomial in $w^{(\beta)}, 1 \leq |\beta| \leq  |\alpha^j|,$  of degree $|\alpha^j|$ with coefficients in the jet space of $h;$  that is why the left hand side of \eqref{Gdetzero} has the form $H(z,w, (w^{(\beta)})_{1 \leq |\beta| \leq k+1})$ and is polynomial in $w^{(\beta)}, 1 \leq |\beta| \leq k+1,$ with coefficients in the $(k+1)$-jet of $F.$ 
Hence (\ref{Gdetzero}) reads as
\begin{equation}\label{HH}
H(z,w, (w^{(\beta)})_{1 \leq |\beta| \leq k+1})=:\eta_0(z,w)h_0((w^{(\beta)})_{1 \leq |\beta| \leq k+1})+ \cdots + \eta_{s}(z,w)h_{s}((w^{(\beta)})_{1 \leq |\beta| \leq k+1})=0.
\end{equation}
 Here $\{h_1, \cdots, h_{s}\}$ is the collection of all distinct monomials in $(w^{(\beta)})_{1 \leq |\beta| \leq k+1}$ of degree at most $d,$ where $d=|\alpha^1|+\cdots+|\alpha^{m+1}|.$ The coefficients $\eta_j(z,w)$ are certain functions, polynomialy depending on $j^{k+1}F$ (the latter dependence is fixed by the choice of $\alpha^{n+1},...,\alpha^{m+1}$).

We prove the following lemma on $\eta_j'$s. Write for each $1 \leq i \leq n,$ we write 
$$\epsilon^i=(0,...,0,1,0,...,0),$$
where the component $``1"$ is at the $i^{\text{th}}$ position, so that $\mathcal{L}^{\epsilon^i}=\mathcal{L}_i, 1 \leq i \leq n.$









\begin{lemma}\label{lemmachm}
We can choose multi-indices $\alpha^i,\, 1 \leq i \leq m+1$  in such a way that
not all $\eta_j,\, 1 \leq j \leq s,$ in (\ref{HH}) are identical zeroes. Moreover, we can achieve
 $\alpha^i=\epsilon^i$ for $1 \leq i \leq n,$
and  $|\alpha^i| \leq i-(n-1)$ for $n+1 \leq i \leq m+1$.
\end{lemma}

\begin{proof}
Suppose, otherwise, that for $\alpha^i =\epsilon^i, 1 \leq i \leq n,$ and any choices of  multi-indices $\alpha^i, n+1 \leq i \leq m+1$ with each $|\alpha^i| \leq i-(n-1),$ we always get all $\eta_j,\, 1 \leq j \leq s,$ identically zero.  This means $H$ in \eqref{HH} is identical zero when $w^{(\beta)}$ are regarded as independent variables, for any such choice of $\alpha^i$'s.

{\bf Claim:} Let $\Gamma$ be any complex manifold in $\mathbb{C}^{n+1}$ defined by $w=h(z)$ passing through the origin, where $h(z)$ is a holomorphic polynomial in $z$ with $h(0)=0, \frac{\partial h}{\partial z_j}(0)=0$ for all $1 \leq j \leq n.$ Then the components of $F$ are linearly dependent over $\mathbb{C}$ on $\Gamma.$

{\bf Proof of Claim:} Write $\Lambda_j=\frac{\partial}{\partial z_j}+ \frac{\partial h(z)}{\partial z_j}\frac{\partial}{\partial w}, 1 \leq j \leq n.$ Note that $\Lambda_j|_0=\frac{\partial}{\partial z_j}.$ Then we have, since $F$ is an embedding,

\begin{equation}\label{eqnlmd}
\mathrm{dim}_{\mathbb{C}}\big(\mathrm{Span}_{\mathbb{C}} \{\Lambda_1 F(q), \cdots, \Lambda_n F(q)\}\big)=n
\end{equation}
for any point $q$ near $0$ on $\Gamma.$

However, by the hypotheses that $H$ (which, we recall, equals to the determinant \eqref{Gdetzero}) is identically zero,  for any choice of  multiindices $\alpha^i, n+1 \leq i \leq m+1$ with $|\alpha^i| \leq i-(n-1)\,\,\forall i,$ we have on $\Gamma$:
\begin{equation}\label{eqnlmdph}
\mathrm{dim}_{\mathbb{C}}\big(\mathrm{Span}_{\mathbb{C}} \{\Lambda_1 F(q), \cdots, \Lambda_n F(q), \Lambda^{\alpha^{n+1}} F(q), \cdots, \Lambda^{\alpha^{m+1}} F(q) \}\big) < m+1.
\end{equation}

We then have the following proposition, which can be regarded as a generalization of Wolsson's result \cite{Wo}.
\begin{proposition}\label{prop42}
Under the assumptions of  (\ref{eqnlmd}) and (\ref{eqnlmdph}) for any choices of  multiindices $\alpha^i, n+1 \leq i \leq m+1$ with each $|\alpha^i| \leq i-(n-1),$ we conclude that there exists  $\lambda_1,...,\lambda_{m+1}$ that are not all zero such that 
$$\lambda _1 \Lambda_j f_1 + \cdots +\lambda_m \Lambda_j f_m + \lambda_{m+1} \Lambda_j g=0$$
on $\Gamma$ for all $1 \leq j \leq n$ at once.
\end{proposition}
\begin{proof}[Proof of \autoref{prop42}]
When $n=1,$ the result follows from the result of Wolsson \cite{Wo}. In the general dimensional case, Proposition \ref{prop42} essentially follows from the framework in the paper of Berhanu and the second author \cite{BX}. To make the paper more self-contained, we include a proof in Appendix II.
\end{proof}
We return to the proof of the Claim. By Proposition \ref{prop42}, the expression
$$\lambda_1 f_1+\cdots+ \lambda_{m} f_m +\lambda_{m+1} g$$ 
is a constant on $\Gamma$ (since all its partial derivatives vanish on $\Gamma$). As we have $F(0)=0$, we finally conclude that the components of $F$ are linearly dependent on $\Gamma.$ This proves the claim. \endpf

\bigskip

\noindent{\it End of the proof of \autoref{lemmachm}.} Now, since $M, F$ are as in the definition of $\tilde{\mathcal{M}}_m,$ the image of $U$ under $F$ is not contained in any affine linear subspace of $\mathbb{C}^{m+1}.$ There exist $m+1$ points $p_j=(a_j^1, \cdots, a_j^n, b_j), 1 \leq j \leq m+1,$ near $0$ such that
\begin{equation}\label{eqnpm1m}
\mathrm{Span}_{\mathbb{C}}\{F(p_1), \cdots, F(p_{m+1})\}=\mathbb{C}^{m+1}.
\end{equation}
Perturbing $p_j'$s if necessary, we can assume that $a_i^1 \neq a_j^1$ if $i \neq j$ and $a_j^1 \neq 0$ for all $1 \leq j \leq m+1.$  Let $h_1(z_1)$ be the holomorphic polynomial in $z_1$ such that $h_1(a_j^1)=b_j, 1 \leq j \leq m+1,  h_1(0)=0.$ Let $h_2(z_1)$ be the holomorphic polynomial in $z_1$ such that $h_2(a_j^1)=1, 1 \leq j \leq m+1, h_2(0)=0.$ Set $h_0(z)=h_1(z_1)h_2(z_1).$ Let $\Gamma_0$ be the complex curve defined by $w=h_0(z).$ Then $\Gamma_0$ satisfies the assumptions in the claim. Moreover, all $p_j, 1 \leq j \leq m+1,$ are on $\Gamma_0.$ We then apply the result in the claim to get a contradiction with (\ref{eqnpm1m}). Thus Lemma \ref{lemmachm} is established.
\end{proof}

In what follows, we assume $\alpha^{1}, \cdots, \alpha^{m+1}$ to be chosen as desired in Lemma \ref{lemmachm}.

On the other hand, $M$ is Levi-nondegenerate at $0$ and thus its Segre family satisfies a completely integrable system of second order PDEs:
\begin{equation}\label{Gpde2}
w''_{ij}=\Phi_{ij}(z,w,w'_{1},...,w'_{n}), \quad 1\leq i,j \leq n,
\end{equation}
for holomorphic near a point $(0,\xi_{1}^0,...,\xi_{n}^0)$ functions
$$\{\Phi_{ij}\},\quad i,j=1,...,n.$$ We now regard both the $(k+1)$-jet prolongation of (\ref{HH}) and the PDE system (\ref{Gpde2}) as submanifolds $\mathcal{M}, \mathcal{E}$ respectively in the $(k+1)-$jet
space $J^{k+1}(\mathbb{C}^n,\mathbb{C})$ with the coordinates $(z,w,\xi_{\alpha})_{1 \leq |\alpha| \leq k+1}$, where $\alpha$ is a multiindex running through all
 $1 \leq |\alpha| \leq k+1$ (here $\xi_{\alpha}$ corresponds to $w^{(\alpha)}$; in particular, $\xi_l$ corresponds to $\frac{\partial w}{\partial z_l}$). Then $\mathcal{M}$ looks as
\begin{equation}\label{Gmanifold1}
H(z,w, (\xi_{\alpha})_{1 \leq |\alpha| \leq k+1})=\eta_0(z,w)h_0((\xi_{\alpha})_{1 \leq |\alpha| \leq k+1})+ \cdots + \eta_{s}(z,w)h_{s}((\xi_{\alpha})_{1 \leq |\alpha| \leq k+1})=0
\end{equation}
and $\mathcal{E}$ as
\begin{equation}\label{Gmanifold2}
\begin{aligned}
\xi_{ij}& =\Phi_{ij}(z, w, \xi);\\
\xi_{ijk}& =(\Phi_{ij})_{z_{k}}+ (\Phi_{ij})_{w}\xi_{k}+ \sum_{l=1}^n (\Phi_{ij})_{\xi_{l}}\Phi_{lk};\\
\cdots&\cdots\\
\xi_{\alpha}& =Q_{\alpha}\Bigl(\xi,(\Phi_{ij}^{(\beta)})_{|\beta| \leq |\alpha|-2}\Bigr);\\
\cdots & \cdots\\
\xi_{\alpha^{m+1}} & =Q_{\alpha^{m+1}}\Bigl(\xi,(\Phi_{ij}^{(\beta)})_{|\beta|\leq k-1}\Bigr)=:\tilde Q\Bigl(\xi,(\Phi_{ij}^{(\beta)})_{|\beta|\leq k-1}\Bigr).
\end{aligned}
\end{equation}
Here  $ \xi=(\xi_{1},...,\xi_{n})$, all $Q_{\alpha}$'s are certain {\em universal} polynomials in their arguments, and we use the notation
$$\xi_{ij}:=\xi_{(0,...,0,1,0,...,0,1,0,...,0)},$$
where $(0,...,0,1,0,...,0,1,0,...,0)$ is a multiindex with the two $1$'s at the $i^{\text{th}}$ and $j^{\text{th}}$ positions, and similarly for $\xi_{ijk}.$

We substitute all $\xi_{\alpha}$'s in $h_j$ by $Q_{\alpha}$ to obtain the  polynomials $h_j$ purely in terms of $\xi$ and $\{\Phi_{ij}\}_{1 \leq i,j \leq n}$
and their derivatives with respect to $z, w, \xi.$
The new polynomials  are denoted by

\begin{equation}\label{Gtilpr}
\tilde h_j(\xi, (\Phi_{ij}^{(\beta)})_{|\beta| \leq k-1}).
\end{equation}





Arguing now as in Section 3 and considering the submanifolds $\mathcal{E}, \mathcal{M}$ of an appropriate jet space corresponding to the PDE systems \eqref{Gmanifold2},\eqref{Gmanifold1} respectively, we write up the fact that $\mathcal{E} \subset \mathcal{M}$ and obtain:
\begin{equation}\label{Gcondition}
\tilde H(z,w,\xi, (\Phi_{ij}^{(\beta)})_{|\beta| \leq k-1}):=\eta_0(z,w)\tilde h_0(\xi, (\Phi_{ij}^{(\beta)})_{|\beta| \leq k-1})+ \cdots +\eta_{s}(z,w)\tilde h_{s}(\xi,  (\Phi_{ij}^{(\beta)})_{|\beta| \leq k-1})=0.
\end{equation}
The condition \eqref{Gcondition} gives us a scalar linear equation for  $\eta_{0}(z,w),...,\eta_{s}(z,w)$ with coefficients depending on $z, w, \xi.$  We then choose a collection of pairwise distinct multiindices
$$\gamma^{1},...,\gamma^{s}\in\bigl(\mathbb Z_{\geq 0}\bigr)^n\quad  \mbox{with}\quad  1\leq |\gamma^{j}| \leq j,$$ and perform $s$ differentiations of the above scalar linear equation
 by means of  the differential operators
$$\frac{\partial^{|\gamma^{1}|}}{\partial \xi^{\gamma^{1}}},...,\frac{\partial^{|\gamma^{s}|}}{\partial \xi^{\gamma^{s}}}.$$ Thus we obtain $s$ new identities
each of which is a scalar linear equation for the functions $\chi_{0},...,\chi_{s},$ and this gives us an $(s+1)\times(s+1)$ linear system. Recall again $\chi_{j}$'s do not all vanish identically, so that the determinant of this system, which we write as
\begin{equation}\label{thedet}\mathcal{D}(\alpha^1,...,\alpha^{m+1} | \gamma^{1},...,\gamma^{s})\left( \{\Phi_{ij}(z,w, \xi)\}_{1 \leq i,j \leq n} \right),
\end{equation}
 vanishes identically, where $\mathcal{D}(\alpha^1,...,\alpha^{m+1} | \gamma^{1},...,\gamma^{s})$ is an differential-algebraic operator. We shall now prove the following
\begin{proposition}\label{multiind}
There exist multiindices $\{{\gamma}^{1},..., {\gamma}^{s}\}$
with $1\leq|{\gamma}^j| \leq j, 1 \leq j \leq s$ such that $\mathcal{D}(\alpha^1,...,\alpha^{m+1} | {\gamma}^{1},...,{\gamma}^{s})$ is not identical zero.
\end{proposition}
\begin{proof}
Similarly as in the case $n=1$, we consider systems of the kind \eqref{Gpde2} with the right hand side depending on the derivatives $w_1',...,w_n'$ only, so that the right hand side in \eqref{Gmanifold2} depends on $\xi$ only (and does not depend on $z,w$). Then we consider the first row in the $(s+1) \times (s+1)$ determinant \eqref{thedet}:
\begin{equation}\label{Growpr}
(\tilde h_{0}(\xi,  (\Phi_{ij}^{(\beta)})_{|\beta| \leq k-1}), \cdots, \tilde h_{s}(\xi,  (\Phi_{ij}^{(\beta)})_{|\beta| \leq k-1})).
\end{equation}

We claim that we can choose analytic functions $\{\Phi_{ij}\}_{1 \leq i,j \leq n}$ in such a way that the components of (\ref{Growpr}) are linearly independent. To prove the claim, we choose a holomorphic function $$w=w^*(z):\,\,(\CC{n},0)\mapsto(\CC{},0)$$ with the following property: {\em $w^*(z)$ does not satisfy any differential-algebraic equation} (the existence of such functions is well known since the classical work of Ostrowski \cite{ostrowski}). 
By moving to a generic point $p$ near $0$ and applying a linear change of coordinates to make $p=0,$ we can assume $(w^*_{z_iz_j}(0))_{1 \leq i,j \leq n}$ is nondegenerate. Then we can express each $z_l$ as a function of $\bigl\{w^*_{z_j}\bigr\}_{j=1}^n$ near $0.$  Next, we choose a complete system of the kind \eqref{Gpde2} (with the defining function $\{\Phi^*_{ij}\}$ depending on $\xi$ only, as discussed above), having $w^*(z)$ as a solution: one can construct $\Phi^*_{ij}$ by expressing, for example, each $z_l$ as a function of $\bigl\{w^*_{z_j}\bigr\}_{j=1}^n$ and substituting the result into $w^*_{z_iz_j}(z)$. 
We then observe that, since $w^*(z)$ is a solution of the system \eqref{Gpde2} with the defining function $\{\Phi^*_{ij}\}$, then evaluating a monomial $h_{j}((w^{(\beta)})_{1 \leq |\beta| \leq k+1}) $ at the $(k+1)$ jet of the function $w=w^*(z)$ amounts (by the definition of $h_j(\xi, (\Phi_{ij}^{(\beta)})_{|\beta| \leq k-1})$) to substituting the $(k+1)$ jet of $w=w^*(z)$ into $h_j(\xi, (\Phi_{ij}^{(\beta)})_{|\beta| \leq k-1})$. Now, assume that for the above choice of the defining function in \eqref{Gpde2} there is a non-trivial linear dependence between the components of the first row of the determinant \eqref{thedet}. Then we conclude that the same  non-trivial linear dependence holds for the monomials  $h_{j}((w^{(\beta)})_{1 \leq |\beta| \leq k+1}) $ evaluated at the $(k+1)$ jet of the function $w=w^*(z)$. Since all the latter monomials are {\em distinct}, we obtain a non-trivial differential-algebraic equation for the function $w=w^*(z)$, which gives a contradiction and proves the claim.

 Now, using the above choice of the defining function in \eqref{Gpde2}, we make use of a result of Wolsson (\cite{Wo}) which states that there exists a non-vanishing identically generalized Wronskians of the components of the first row, and this yields the existence of the desired multiindices $\{{\gamma}^{1},...,{\gamma}^{s}\}.$
\end{proof}

We have now an important

\begin{remark}\label{integrable}
In fact, the proof of \autoref{multiind} implies a stronger fact, which is the non-triviality of the restriction of the operator $\mathcal D$ constructed in \autoref{multiind} onto the subset $I$ of all possible analytic right hand sides $(\{\Phi_{ij}\}_{1 \leq i,j \leq n})$ {\em corresponding to  completely integrable systems \eqref{Gpde2}}. 
\end{remark}

Further, we emphasize the following.

\begin{remark}
The operator $\mathcal{D}(\alpha^1,...,\alpha^{m+1} | \gamma^{1},...,\gamma^{s})$ constructed in \autoref{multiind} is universal, in the sense that it depends on $\{\alpha^{1},...,\alpha^{m+1}\}$ and $\{\gamma^{1},...,\gamma^{s}\}$ only. In turn,
$\{{\gamma}^{1},...,{\gamma}^{s}\}$ are determined by $\{\alpha^{1},...,\alpha^{m+1}\},$ that is why we write in short
$$\mathcal{D}(\alpha^{1},...,\alpha^{m+1})=\mathcal{D}(\alpha^1,...,\alpha^{m+1} | {\gamma}^{1},...,{\gamma}^{s})$$ in what follows.
We also remark that  for each $\{\alpha^1,...,\alpha^{m+1}\},$ the order of derivatives of $\{\Phi_{ij}\}_{1 \leq i,j \leq n}$ that appears in $\mathcal{D}(\alpha^1,...,\alpha^{m+1})$
is at most $s+k-1 \leq s + m-n.$ Indeed, this can be easily seen from (\ref{Gtilpr}) and
the fact that $|{\gamma}^j| \leq s$ for all $1 \leq j \leq s$ in \autoref{multiind}. Thus, the order $d$ of the  differential-algebraic operator $\mathcal{D}(\alpha^{1},...,\alpha^{m+1})$  satisfies
\begin{equation}\label{order}
d\leq  s + m-n.
\end{equation}
\end{remark}

Now to obtain a differential-algebraic operator annihilating the right hand side of any nondegenerately embeddable hypersurface, we argue as follows. For a holomorphic nondegenerate embedding map $F: M \rightarrow \mathcal{Q}^{2m+1}$ defined near $0$ we consider all possible choices of
 $\{\alpha^{1},...,\alpha^{m+1}\},$ where they are as required in Lemma \ref{lemmachm}.   In particular there exist finitely many such choices of $\{\alpha^{1},...,\alpha^{m+1}\}$.  Then we obtain a collection of finitely many operators
$$\{\mathcal{D}_{1},..., \mathcal{D}_{\nu}\}$$
such that for any $M \subset \tilde{\mathcal{M}}_{m},$ there exists $l\leq\nu$ with $\mathcal D_l(\{\Phi_{ij}\}_{1 \leq i, j \leq n})=0$. Now the product operator
\begin{equation}\label{prodop}
\mathcal{D}(\{\Phi_{ij}\}_{1 \leq i,j \leq n}):=\mathcal{D}_{1}(\{\Phi_{ij}\}_{1 \leq i,j \leq n}) \cdots \mathcal{D}_{\nu}(\{\Phi_{ij}\}_{1 \leq i,j \leq n}).
\end{equation}
satisfies $\mathcal{D}(\{\Phi_{ij}\}_{1 \leq i,j \leq n}) \equiv 0$ if $\{\Phi_{ij}\}$ is associated to some $M \subset \tilde{M}_{m}.$

 The next step in proving \autoref{T4} is to transfer to the complex defining function $\rho(z,a,b)$, which we do similarly to the case $n=1$. Our goal is to express $\mathcal D(\{\Phi_{ij}\}_{1 \leq i, j \leq n})$   as a rational  function of the $(d+2)$-jet of $\rho$ (where $d$ is the order of $\mathcal D$).

In view of the Levi-nondegeneracy of $M$ near $0,$ the map
$$(z,a,b) \mapsto (z, \rho(z,a,b), \rho_{z}(z,a,b))$$
is a local biholomorphism between $\mathbb{C}^{2n+1}(z,a,b)$ and $\mathbb{C}^{2n+1}(z,w, \xi).$ Here we write $z=(z_{1},...,z_{n}), a=(a_{1},...,a_{n}).$ If
$$(z,w,\xi) \mapsto (z,A(z,w,\xi), B(z,w,\xi))$$
is the inverse biholomorphism, where we write $A(z,w,\xi)=(A_{1}(z,w,\xi),...,A_{n}(z,w,\xi)),$
then we have
\begin{equation}\label{e1}
w=\rho(z,a,b), \xi_{i}=\rho_{z_{i}}(z,a,b), \Phi_{ij}(z,w,\xi)=\rho_{z_{i}z_{j}}(z,a,b), 1 \leq i,j \leq n,
\end{equation}
where $a=A(z,w,\xi),b=B(z,w,\xi).$
Thus $\Phi_{ij}$ is already expressed in terms of the derivatives of $\rho.$ Then let us demonstrate the way
to express, for instance, $(\Phi_{ij})_{z_{k}}, 1 \leq k \leq n.$
First,
\begin{equation}\label{GPhiw}
(\Phi_{ij})_{z_{k}}=\rho_{z_{i}z_{j}z_{k}}+ \sum_{l=1}^{n} \rho_{z_{i}z_{j}a_{l}}(A_{l})_{z_{k}} + \rho_{z_{i}z_{j}b}B_{z_{k}}.
\end{equation}
For each fixed $k,$ we now need to compute $(A_{l})_{z_{k}}, B_{z_{k}}, 1 \leq l \leq n,$ in terms of $\rho$ and its derivatives. For that we differentiate the first two equations in \eqref{e1}
with respect to $z_{k}$ and get,
\begin{equation}\label{GsystemAB}
\begin{split}
0 & =\rho_{z_{k}}+\sum_{l=1}^{n}\rho_{a_{l}}(A_{l})_{z_{k}}+\rho_{b}B_{z_{k}}\\
0 & =\rho_{z_{i}z_{k}}+\sum_{l=1}^{n}\rho_{z_{i}a_{l}}(A_{l})_{z_{k}}+\rho_{z_{i}b}B_{z_{k}}, 1 \leq i \leq n
\end{split}
\end{equation}

Note that \eqref{GsystemAB} is a linear system for $(A_{l})_{z_{k}}, B_{z_{k}}, 1 \leq l \leq n$ whose determinant at the reference point $(0,0,\xi^0)$ being equal to the Levi determinant of $M$ at the origin. By the Levi-nondegeneracy and applying Cramer's rule, we can solve $(A_{l})_{z_{k}}, B_{z_{k}}, 1 \leq l \leq n$ as  rational functions of $2-$jet
of $\rho.$ Substituting into \eqref{GPhiw}, we get $(\Phi_{ij})_{z_{k}}$ as rational functions of the $3-$jet of $\rho.$ We similarly express the entire $d-$jet of  $\Phi_{ij}$
as a rational  function of the $(d+2)-$jet of $\rho$.
In this manner we obtain an algebraic differential operators $D^C,$ such that the condition
$\mathcal{D}(\{\Phi_{ij}\}_{1 \leq i,j \leq n})\equiv 0$ reads as $D^{C}(\rho(z,a,b)) \equiv 0.$

Furthermore, we can apply an  argument similar to that in Section 3 to find an differential-algebraic operator $D=D(n,m)$ of order $d+2$ such that $D^C(\rho) \equiv 0$ reads as $D(\varphi) \equiv 0,$ where $\varphi(z,a,u)$ is the real defining function of a hypersurface $M$.

Finally, it remains to show that  $D$ is non-trivial on the space of real-analytic defining functions $\varphi(z,a,u)$ of real hypersurfaces. Since the latter
subspace is totally real, it is enough to show the non-triviality of $D$ on the space of all possible analytic $\varphi(z,a,u)$, which is equivalent to the non-triviality of $D^C(\rho)$ on the space of all possible analytic $\rho(z,a,b)$. The desired non-triviality of $D^C(\rho)$ amounts to the non-triviality of $\mathcal D$  on the subspace $I$ of all possible analytic right hand sides $(\{\Phi_{ij}\}_{1 \leq i,j \leq n})$ {corresponding to  completely integrable systems \eqref{Gpde2}} and hence follows from \autoref{integrable}. This completes the proof of \autoref{T4}. \qed

It is not difficult now to verify the proof of \autoref{T1}.

\begin{proof}[Proof of \autoref{T1}] Recall that we have the decomposition \eqref{union}. Thus the desired differential-algebraic operator is given as the product
\begin{equation}\label{theoperator}
D(n,n)(\varphi)\cdots D(n,N)(\varphi).
\end{equation}
 As follows from the construction, \eqref{theoperator} annihilates real defining functions of all hypersurfaces from $\mathcal{M}_{N}$ and is shift-invariant. This completes the proof of \autoref{T1}.
 \end{proof}

Arguing then identically to the proof of \autoref{T2} in the case $n=1$, we obtain  the proof of \autoref{T5} with $\nu(n,m)$ being the order of the differential operator in \autoref{T4}.  We shall note that, as follows from  \eqref{prodop}, the differential operator in \autoref{T1} is obtained by multiplying several lower order operators. Using this observation, we can improve the bound for $\nu(n,m)$ to being equal to
\begin{equation}\label{maxorder}
\nu(n,m)=2+\mbox{max}\,\Bigl\{\mbox{ord}\,\mathcal D(\alpha^1,..,\alpha^m)\Bigr\},
\end{equation}
where $\alpha^1,...,\alpha^m\in\bigl(\mathbb{Z}_{\geq 0}\bigr)^n$ are all distinct and  satisfy the requirement in Lemma \ref{lemmachm}. More precisely, we require for $1 \leq i \leq n, \alpha^i=\epsilon^i,$
and for $n+1 \leq i \leq m+1, |\alpha^i| \leq i-(n-1).$

It is also immediate to see from \eqref{theoperator} and \eqref{union} that \autoref{T5} implies \autoref{T2} with
$\mu(n,N)$ being equal to
\begin{equation}\label{maxorder1}
\mu(n,N)=\mbox{max}\,\bigl\{\nu(n,m)\bigr\}_{m\in[n,N]}~.
\end{equation}

Explicit bounds  for $\nu(n,m),\mu(n,m)$ are given in Appendix I below.

In the end of this section, we would like to emphasize that it looks particularly interesting to study, by using the method of the present paper, the problem of characterization of the class of real hypersurfaces that can be holomorphically embedded into the particular hyperquadric $\mathbb{S}^{2N+1}$ (namely, the sphere). On the latter problem, see \cite{HZ},\cite{es1}, \cite{HLX}, \cite{HX}. In particular, \cite{HX}  shows that the examples of compact real algebraic strictly pseudoconvex hypersurfaces constructed in \cite{HLX} are not holomorpically embeddable into any spheres.  We conjecture here that the property of the embeddability into a sphere is characterized by differential-algebraic inequalities, supplementing the differential-algebraic equations introduced in the paper. (Such inequalities should arise from the Cramer's rule applied to a linear system of the kind \eqref{Gsegrealpha}).

\section{Appendix I}
In this section we  obtain explicit upper bounds for  $\nu(n,m),\mu(n,m)$ in \autoref{T5} and \autoref{T2} respectively.

In view of \eqref{maxorder},\eqref{maxorder1} and \eqref{order}, obtaining upper bounds for  $\nu(n,m),\mu(n,m)$ amounts to  estimating, for each $\{\alpha^1,...,\alpha^m\}$ as above, the total number $s$ of terms in \eqref{Gcondition}.  We first introduce,
\begin{definition}
Let $\lambda w^{(\beta^1)}...w^{(\beta^l)}, l\geq 1, $ be a monomial in the derivatives of $w$. We define the {\em weighted degree} of this monomial to be $|\beta^1|+...+|\beta^l|.$ If $W$ ia a polynomial which is a sum of monomials of such form, then the weighted degree of $W$ is defined to be the highest
weighted degree of these monomials.
\end{definition}
\begin{lemma}
Let $\mathcal{L}_{j}, 1 \leq j \leq n$ be as in Section 4. Let $h$ be one of the functions $f_{1},...,f_{m},g.$ For any multiindex $\alpha,$
$\mathcal{L}^{\alpha}h$ is a polynomial in $\{w^{(\beta)}\}_{|\beta|\leq |\alpha|}$ of weighted degree $|\alpha|.$
\end{lemma}
\begin{proof}
Note that
$$\mathcal{L}_{j}h = \frac{\partial h}{\partial z_{j}}+ w'_{j}\frac{\partial h}{\partial w};$$
$$\mathcal{L}_{i}\mathcal{L}_{j}h=\frac{\partial^2 h}{\partial z_{i}\partial z_{j}}+w'_{i}\frac{\partial^2 h}{\partial z_{j} \partial w}+
 w'_{j}\frac{\partial^2 h}{\partial z_{i} \partial w} +w'_{i}w'_{j}\frac{\partial^2 h}{\partial w^2}+ w''_{ij}\frac{\partial h}{\partial w}.$$
 Hence the conclusion holds for $n=1$ and $n=2.$ The general case can be proved by
induction.
\end{proof}

\begin{lemma}\label{Gwdegpq}
$H$ in (\ref{HH}) is a polynomial in $\{w^{(\beta)}\}_{|\beta| \leq k+1}$ of
weighted degree at most $\frac{(m+1)(m+2)}{2}$ with coefficients in $j^{k+1}F.$
\end{lemma}
\begin{proof}
Note that $|\alpha^i| \leq i$ for each $1 \leq i \leq m+1.$ Thus $\mathcal{L}^{\alpha^i}f_{1},...,\mathcal{L}^{\alpha^i}f_{m}, \mathcal{L}^{\alpha^i}g$ are all polynomials in $\{w^{(\beta)}\}_{1 \leq |\beta| \leq i}$ of weighted order at most $i$ with coefficients in $j^i F.$ Then the statement  follows by an easy computation from its definition (\ref{Gdetzero}).
\end{proof}

We now need to convert $H(z,w,\{\xi_{\beta}\}_{1 \leq |\beta| \leq k+1}),$ where $\xi_{\beta}$ corresponds to $w^{(\beta)},$ to $\tilde{H}(z,w,(\Phi_{ij}^{(\beta)})_{|\beta| \leq k-2}).$ For that we study $\xi_{\beta}, |\beta| \geq 2$ as a polynomial in $\xi_{1},...,\xi_{n}, \{\Phi_{ij}^{(\gamma)}\}_{0 \leq |\gamma| \leq |\beta|-2}.$

Recall
 \begin{equation}
 \xi_{\beta}=Q_{\beta}(\xi_{1},...,\xi_{n}, \{\Phi_{ij}^{(\gamma)}\}_{0 \leq |\gamma| \leq |\beta|-2}), |\beta| \geq 2.
 \end{equation}
 where $Q_{\beta}$ is a polynomial in its argument.

\begin{lemma}\label{Gdegree}
Let $Q_{\beta}, |\beta| \geq 2$ be as above. Then $Q_{\beta}$ is a polynomial of degree $\leq |\beta|-1$
in its arguments.
\end{lemma}
\begin{proof}
When $|\beta|=2,$ the statement is trivial since
$$\xi_{ij}=\Phi_{ij}, 1 \leq i,j \leq n.$$
The case $|\beta|=3$ is verified as follows.
\begin{equation}
\begin{split}
\xi_{ijk}& =(\Phi_{ij})_{z_{k}}+ (\Phi_{ij})_{w}\xi_{k}+ \sum_{l=1}^n (\Phi_{ij})_{\xi_{l}}\xi_{lk}\\
& = (\Phi_{ij})_{z_{k}}+ (\Phi_{ij})_{w}\xi_{k}+ \sum_{l=1}^n (\Phi_{ij})_{\xi_{l}}\Phi_{lk}, 1 \leq i,j ,k \leq n.
\end{split}
\end{equation}
Then general case can be proved by induction.
\end{proof}
\autoref{Gdegree} leads to the following lemma.
\begin{lemma}\label{Gdegtil}
$\tilde{H}(z,w,\xi, (\Phi_{ij}^{(\beta)})_{|\beta| \leq k-1})$ are polynomials of degrees $\leq \frac{(m+1)(m+2)}{2}$ in the arguments $\xi,\Phi_{ij}^{(\beta)}$.
\end{lemma}
\begin{proof}
Write any monomial of $H$ in the following form:
$$h(s,t)w^{(\beta^1)}...w^{(\beta^{\tau})},$$
with $|\beta^1|+...+|\beta^\tau|\leq \frac{m(m+1)}{2}$ by \autoref{Gwdegpq}.
Now each $w^{(\beta^i)},$ by \autoref{Gdegree}, can be written as a polynomial $Q_{\beta^i}(\xi_{1},...,\xi_{n}, \{\Phi_{ij}^{(\gamma)}\}_{0 \leq |\gamma| \leq |\beta^i|-2})$ of degree $\leq |\beta^i|$ if $|\beta^i|\geq 2.$ The statement of \autoref{Gdegtil} then follows easily.
\end{proof}

\begin{lemma}

{\bf (1).} For each $l \geq 0,$ there are $\left(\begin{array}{c}
                                           l+n-1 \\
                                           n-1 \\
                                         \end{array}\right)$
distinct multiindices $\beta$ such that $|\beta|=l.$
Here we let $\left(
                 \begin{array}{c}
                   0 \\
                   0 \\
                 \end{array}
               \right)=1.$
 Consequently, $\{\xi, (\Phi_{ij}^{(\beta)})_{|\beta| \leq k-1}\}$ has
  \begin{equation}\label{Gterms}
  n+ \frac{n(n+1)}{2}\left(
                 \begin{array}{c}
                   n+k-1 \\
                   n \\
                 \end{array}
               \right)
\end{equation}
terms. 
Moreover, since $k \leq m-n+1,$ we have (\ref{Gterms})  bounded by
\begin{equation}\label{pmn}
p(m,n):=n+\frac{n(n+1)}{2}\left(
            \begin{array}{c}
              m \\
              n \\
            \end{array}
          \right).
\end{equation}

\begin{proof}
The proof of the lemma follows from elementary combinatorics.
\end{proof}

\end{lemma}

\bigskip

We are now able, using elementary combinatorics again, to give an estimate for the number of terms in $\tilde{H}$ and thus for the integer $s$. Note $\tilde{H}$ is a polynomial of degree at most $(m+1)(m+2)/2$ with at most $p(m,n)$ variables.

\begin{proposition}\label{bounds} One has
\begin{equation}
s \leq \left(
                                    \begin{array}{c}
                                     (m+1)(m+2)/2+p(m,n) \\
                                      p(m,n) \\
                                    \end{array}
                                  \right)                                 
\end{equation}
Here $p(m,n)$ is defined by \eqref{pmn}.
\end{proposition}
Combining now \autoref{bounds} with \eqref{order} and \eqref{maxorder}, we get
\begin{theorem}\label{thebound}
The integers $\nu(n,m)$ and $\mu(n,N)$ in \autoref{T4} and \autoref{T1} respectively can be explicitly chosen to be
\begin{equation}\label{maxnu}\nu(n,m):= 2+m-n+
                               \left(        \begin{array}{c}
                                    (m+1)(m+2)/2+p(m,n) \\
                                      p(m,n) \\
                                    \end{array}
                                  \right)
\end{equation}

\begin{equation}\label{maxorder2}
\mu(n,N):=\nu(n,N).
\end{equation}
Here $p(m,n)$ is the explicit expression given by \eqref{pmn}.

\end{theorem}

\begin{proof}
\autoref{bounds} and formulas \eqref{order}, \eqref{maxorder} immediately imply \eqref{maxnu}.
In order to prove \eqref{maxorder2}, we note that the expression
$$\left(        \begin{array}{c}
                                    (m+1)(m+2)/2+p(m,n) \\
                                      p(m,n) \\
                                    \end{array}
                                  \right)$$
 is monotonous in $m$ for fixed $n$. Indeed, if $m\leq m'$, set
$c:=(m+1)(m+2)/2,\,d:= p(m,n),\, c':=(m'+1)(m'+2)/2,\,d':= p(m',n)$. We have $c\leq c',\,d\leq d'$. Then
$$\left( \begin{array}{c}
                                    c+d \\
                                      d \\
                                    \end{array}
                                  \right)
                                  \leq
                                  \left( \begin{array}{c}
                                    c'+d \\
                                      d \\
                                    \end{array}
                                  \right)
                                  =
                                  \left( \begin{array}{c}
                                    c'+d \\
                                      c' \\
                                    \end{array}
                                  \right)
                                  \leq
                                  \left( \begin{array}{c}
                                    c'+d' \\
                                      c' \\
                                    \end{array}
                                  \right)
                                  =
                                  \left( \begin{array}{c}
                                    c'+d' \\
                                      d' \\
                                    \end{array}
                                  \right),$$
as required for the monotonicity. Now \eqref{maxorder2} follows from \eqref{maxorder1}.

\end{proof}

\section{Appendix II}
In this section, we provide a brief proof of Proposition \ref{prop42}. We first introduce following notions and definitions.
Let $\mathcal{M}$ be a $n-$dimensional (connected) complex manifold. Write $\Lambda_1, \cdots, \Lambda_m$ as a basis of holomorphic vector field of $\mathcal{M}.$ In particular, we assume $\Lambda_j h, 1 \leq j \leq n,$ is holomorphic whenever $h$ is holomorphic.  As before, for a multiple index $\alpha=(\alpha_1, \cdots, \alpha_n),$ write 
$\Lambda^{\alpha}=\Lambda_1^{\alpha_1}...\Lambda_n^{\alpha_n}.$ Let $H=(h_1, \cdots, h_N)$ be a holomorphic map from $\mathcal{M}$ to $\mathbb{C}^N, N \geq n.$

\begin{definition}\label{defnek}
For each $l \geq 1, q  \in \mathcal{M},$ define
$$E_l(q):=\mathrm{Span}_{\mathbb{C}}\{\Lambda^{\alpha} H(q): 1 \leq |\alpha| \leq l\}.$$
\end{definition}

\begin{remark}
It is easy to see that if $H$ is an embedding at $q,$ then $\mathrm{dim}_{\mathbb{C}}\left(E_1(q)\right)=n.$
\end{remark}
To establish Proposition \ref{prop42}, we need the following result.

\begin{theorem}\label{theoremlambda}
Let $O$ be an open subset of $\mathcal{M}.$ Let $l \geq 1, n \leq m < N.$ Assume that $\mathrm{dim}_{\mathbb{C}}\left(E_1(q)\right)=n, \mathrm{dim}_{\mathbb{C}}\left( E_{l}(q)\right)=\mathrm{dim}_{\mathbb{C}}\left( E_{l+1}(q)\right)=m$ for any $q \in O.$ Then there exists complex numbers $\lambda_1, \cdots \lambda_N$ that are not all zero, such that,
$$\lambda_1\Lambda_j h_1 + \cdots +\lambda_N  \Lambda_j h_N=0,~\text{for any}~1 \leq j \leq n.$$
\end{theorem}
\begin{proof}
The theorem basically follows from a similar argument as in \cite{BX} (See also \cite{FHX}). We sketch a proof here. By the assumption that $\mathrm{dim}_{\mathbb{C}}\left( E_{l}(q)\right)=m$, shrinking $O$ if necessary, there exist multi-indices $\alpha^1,\cdots, \alpha^m$ with each $|\alpha^i| \leq l.$ such that
\begin{equation}\label{eqnal1}
\mathrm{dim}_{\mathbb{C}}\left(\mathrm{Span}_{\mathbb{C}}\{\Lambda^{\alpha^1} H(q),..., \Lambda^{\alpha^{m}}H(q) \}\right)=m~\text{for every}~q \in O.
\end{equation}
Since $\mathrm{dim}_{\mathbb{C}}\left(E_1(q)\right)=n,$ we can choose in (\ref{eqnal1}) $\alpha^i=\epsilon^i, 1 \leq i \leq n.$ Here we write 
for each $1 \leq i \leq n, \epsilon^{i}=(0,...0,1,0,...,0),$ where the component $``1"$ is at the $i^{\text{th}}$ position, so that $\mathcal{L}^{\epsilon^i}=\mathcal{L}_i, 1 \leq i \leq n.$

The assumption that $\mathrm{dim}_{\mathbb{C}}\left( E_{l+1}(q)\right)=m$ implies that for any multiindex $\beta$ with $|\beta| \leq l+1.$
\begin{equation}\label{eqnalbe}
\mathrm{dim}_{\mathbb{C}}\left(\mathrm{Span}_{\mathbb{C}}\{\Lambda^{\alpha^1} H(q),..., \Lambda^{\alpha^{m}}H(q), \Lambda^{\beta} H(q) \}\right) =m~\text{for every}~q \in O.
\end{equation}

By equation (\ref{eqnal1}), we conclude that there exists $j_1, j_2, \cdots, j_m$ such that, by shrinking $O$ if necessary,
\begin{equation}
\left|
\begin{array}{llll}
\Lambda^{\alpha^1} h_{j_1} &  ... & \Lambda^{\alpha^1} h_{j_m} \\
... & ... & ... \\
\Lambda^{\alpha^m} h_{j_1} &  ... & \Lambda^{\alpha^m} h_{j_m}\\
\end{array} \right| \neq 0~\text{at every}~q \in O.
\end{equation}
To make the notations simple, we assume, without loss of generality, that $j_1=1, \cdots, j_m=m.$ That is,
\begin{equation}\label{eqnalphah}
\left|
\begin{array}{llll}
\Lambda^{\alpha^1} h_{1} &  ... & \Lambda^{\alpha^1} h_{m} \\
... & ... & ... \\
\Lambda^{\alpha^m} h_{1} &  ... & \Lambda^{\alpha^m} h_{m}\\
\end{array} \right| \neq 0~\text{at every}~q \in O.
\end{equation}

We conclude by equation
(\ref{eqnalbe}) that for any multiindex $\beta$ with $|\beta| \leq l+1.$
\begin{equation}\label{eqnbeta}
\left|\begin{matrix}
\Lambda^{\alpha^{1}}h_{1} & ... & \Lambda^{\alpha^{1}} h_{m}  & \Lambda^{\alpha^{1}} h_{m+1} \\
... & ... & ... & ... \\
\Lambda^{\alpha^{m}} h_1 & ... & \Lambda^{\alpha^{m}} h_{m} & \Lambda^{\alpha^{m}} h_{m+1}  \\
\Lambda^{\beta} h_1 & ... & \Lambda^{\beta} h_m  & \Lambda^{\beta} h_{m+1}
\end{matrix}\right| \equiv 0~\text{for every}~q \in O.
\end{equation}

\begin{lemma}\label{lemmad0}
For any $1 \leq \nu \leq n,$ and $i_1 < i_2 < \cdots < i_{m-1}$ with $\{i_1, \cdots,i_{m-1} \} \subset\{1, 2, \cdots, m \}$,
the following holds:
\begin{equation}\label{eqnquo}
\Lambda_{\nu}\left(\frac{\left|\begin{matrix}
	\Lambda^{\alpha^{1}}h_{i_1} & ... & \Lambda^{\alpha^{1}} h_{i_{m-1}}  & \Lambda^{\alpha^{1}} h_{m+1} \\
	... & ... & ... & ... \\
	\Lambda^{\alpha^m} h_{i_1} & ... & \Lambda^{\alpha^m} h_{i_{m-1}}  & \Lambda^{\alpha^m} h_{m+1}
	\end{matrix}\right|}{\left|\begin{matrix}
	\Lambda^{\alpha^{1}}h_{1} & ... & \Lambda^{\alpha^{1}} h_{m-1}  & \Lambda^{\alpha^{1}} h_{m} \\
	... & ... & ... & ... \\
	\Lambda^{\alpha^m} h_1 & ... & \Lambda^{\alpha^m} h_{m-1}  & \Lambda^{\alpha^m} h_{m}
	\end{matrix}\right|} \right) \equiv 0.
\end{equation}
\end{lemma}
\begin{proof}
The conclusion follows from (\ref{eqnbeta}). Indeed,  the numerator of the left hand side of (\ref{eqnquo}) can be written as a summation of terms that are multiples of the left hand side of (\ref{eqnbeta}) for certain choices of $\beta$. A detailed proof can be copied from page 1391 of \cite{BX}.
\end{proof}

Lemma \ref{lemmad0} implies that the function in the big parentheses in the equation (\ref{eqnquo}) is a constant in $O.$
We now fix some notations. If $i_1 < \cdots < i_{m-1}$ and $(i_1, \cdots, i_{m-1})=(1,2,...,\hat{i_0}, ...,m)$(Here $(1,2,...,\hat{i_0}, ...,m)$ means $(1,2,...,m)$ with the component $``i_0"$ missing.), then we write the constant
$$c_{i_0}:=\frac{\left|\begin{array}{cccc}
	\Lambda^{\alpha^1} h_{i_1} & \cdots & \Lambda^{\alpha^1} h_{i_{m-1}} & \Lambda^{\alpha^1} h_{m+1} \\
	\cdots & \cdots & \cdots & \cdots \\
	\Lambda^{\alpha^m} h_{i_1} & \cdots & \Lambda^{\alpha^m} h_{i_{m-1}} & \Lambda^{\alpha^m} h_{m+1}
	\end{array}\right|}{\left|\begin{matrix}
	\Lambda^{\alpha^{1}}h_{i_1} & ... & \Lambda^{\alpha^{1}} h_{i_{m-1}}  & \Lambda^{\alpha^{1}} h_{i_0} \\
	... & ... & ... & ... \\
	\Lambda^{\alpha^m} h_{i_1} & ... & \Lambda^{\alpha^m} h_{i_{m-1}}  & \Lambda^{\alpha^m} h_{i_0}
	\end{matrix}\right|}$$
Consequently,
\begin{equation}\label{eqnla}
\left|\begin{matrix}
\Lambda^{\alpha^{1}}h_{i_1} & ... & \Lambda^{\alpha^{1}} h_{i_{m-1}}  & \Lambda^{\alpha^{1}} (h_{m+1}-c_{i_0}h_{i_0}) \\
... & ... & ... & ... \\
\Lambda^{\alpha^m} h_{i_1} & ... & \Lambda^{\alpha^m} h_{i_{m-1}}  & \Lambda^{\alpha^m} (h_{m+1}-c_{i_0} h_{i_0})
\end{matrix}\right| \equiv 0~\text{in}~O.
\end{equation}

Since $i_0$ may vary from $1$ to $m$, we thus have $m$ constants: $c_1, ...,c_m.$ We now prove the following lemma.

\begin{lemma}\label{lemma64}
The following holds in $O$ for any $i_1 <i_2< ...< i_{m-1}$ with $\{i_1,...,i_{m-1}\} \subset \{1,2,...,m\}:$
$$\left|\begin{matrix}
\Lambda^{\alpha^{1}}h_{i_1} & ... & \Lambda^{\alpha^{1}} h_{i_{m-1}}  & \Lambda^{\alpha^{1}} (h_{m+1}-\sum_{i=1}^m c_{i}h_{i}) \\
... & ... & ... & ... \\
\Lambda^{\alpha^m} h_{i_1} & ... & \Lambda^{\alpha^m} h_{i_{m-1}}  & \Lambda^{\alpha^m} (h_{m+1}-\sum_{i=1}^m c_{i} h_{i})
\end{matrix}\right| \equiv 0~\text{in}~O.$$
\end{lemma}
\begin{proof}
Assume that $(i_1,...,i_{m-1})=(1,2,...,\hat{i_0},...,m).$ Note that if $i \neq i_0,$ i.e., $i \in \{i_1,...,i_{m-1}\},$
then 
\begin{equation}\label{eqnhi}
\left|\begin{matrix}
\Lambda^{\alpha^{1}}h_{i_1} & ... & \Lambda^{\alpha^{1}} h_{i_{m-1}}  & \Lambda^{\alpha^{1}} (c_{i}h_{i}) \\
... & ... & ... & ... \\
\Lambda^{\alpha^m} h_{i_1} & ... & \Lambda^{\alpha^m} h_{i_{m-1}}  & \Lambda^{\alpha^m} (c_{i} h_{i})
\end{matrix}\right| \equiv 0.
\end{equation}
Indeed, the last column of the above matrix is just a constant multiple of one of the first $m-1$ columns. Then Lemma \ref{lemma64} follows easily from equations (\ref{eqnla}) and (\ref{eqnhi}).
\end{proof}

We recall the following lemma from \cite{BX}.
\begin{lemma}\label{lemmabx}
Let ${\bf{b}}_{1},\cdots,{\bf{b}}_{n}$ and ${\bf{a}}$ be  $n$-dimensional column vectors with elements in $\mathbb{C}$, and let $B=({\bf{b}}_{1},\cdots,{\bf{b}}_{n})$ denote the $n\times n$ matrix.  Assume that $\mathrm{det}B \neq 0,$ and that
$\mathrm{det}({\bf{b}}_{i_{1}},{\bf{b}}_{i_{2}},\cdots,{\bf{b}}_{i_{n-1}},{\bf{a}})=0$ for any $1 \leq i_{1} < i_{2} < \cdots < i_{n-1} \leq n.$ Then $\bf{a}=0,$ where $\bf{0}$ is the $n$-dimensional zero column vector.
\end{lemma}

By Lemmas \ref{lemma64},~\ref{lemmabx}, and equation (\ref{eqnalphah}), we conclude that 
$$\Lambda^{\alpha^j}(h_{m+1}-\sum_{i=1}^m c_i h_i)=0,~\forall 1 \leq j \leq m.$$
In particular, when $1 \leq j \leq n, $ since $\Lambda^{\alpha^j}=\Lambda_j,$ We thus conclude that
$$\Lambda_j h_{m+1}-\sum_{i=1}^m c_i \Lambda_j h_i=0,~1 \leq j \leq n. $$
This establishes Theorem \ref{theoremlambda}.
\end{proof}

\begin{remark}\label{remark6q}
We state the following fact which is an immediate consequence of Theorem \ref{theoremlambda}. With the notions in Definition \ref{defnek}, assume that there do not exist
constants $\lambda_1,...,\lambda_N$ that are not all zero such that
$$\lambda_1 \Lambda_j h_1+ \cdots+ \lambda_N \Lambda_j h_N=0.$$
Let $l \geq 1, O$ an open subset of $\mathcal{M}.$ Assume $\mathrm{dim}_{\mathbb{C}}\left(E_1(q)\right)=n, \mathrm{dim}_{\mathbb{C}}\left(E_l(q) \right)=m < N,$
for any $q \in O$. Then for a generic $\widetilde{q} \in O,$ we have $E_l(\widetilde{q}) \subsetneqq E_{l+1}(\widetilde{q}).$
\end{remark}

Remark \ref{remark6q} implies Proposition \ref{prop42}.


\begin{thebibliography}{99999}


\bibitem{ber}  M. S. Baouendi, P. Ebenfelt, L. P. Rothschild.
{\it Real Submanifolds in Complex Space and Their Mappings},
Princeton University  Press, Princeton Math. Ser. {\bf 47},
Princeton, NJ, 1999.

\bibitem{bh} S. Baouendi and X. Huang. {\it Super-rigidity for holomorphic mappings between hyperquadrics with positive signature}, J. Differential Geom. {\bf 69} (2005), no. 2, 379-398.

\bibitem{beh} S. Baouendi, P. Ebenfelt, and X. Huang. {\it Super-rigidity for CR embedding of real hypersurfaces into hyperquadrics}, Adv. Math. {\bf 219} (2008), no. 5, 1427-1445.

\bibitem{BX} S. Berhanu and M. Xiao, {\it On the $C^{\infty}$ version of the reflection principle for mappings between CR manifolds}, Amer. J. Math. 137 (2015), 1365-1400.


\bibitem{cartan} E. Cartan. "{\em Sur la geometrie pseudo-conforme des hypersurfaces de l'espace de deux
variables complexes II},  Ann. Scuola Norm. Sup. Pisa Cl. Sci. (2)
{\bf 1}  (1932), no. 4, 333-354.

\bibitem{cartanODE} E. Cartan. {\em Sur les vari\'et\'es \'a connexion projective}, Bull.
Soc. Math. France {\bf 52} (1924), 205-241.

\bibitem{CM} S. Chern and J. Moser.  {\it Real hypersurfaces in complex manifolds}, Acta Math. {\bf 133} (1974), no. 1, 219-271.

\bibitem{Da} J. D'Angelo, {\it Hermitian analogues of Hilbert's 17-th problem}, Adv. Math.{\bf 226} (2011), no. 5, 4607-4637.

\bibitem{Da1} J. D'Angelo, {\it CR complexity and hyperquadric maps}, Analysis and Geometry, Springer Proceedings in Mathematics $\&$ Statistics 127, (2015),  17-34.




\bibitem{DL2} J. D'Angelo and J. Lebl, {\it Hermitian symmetric polynomials and CR complexity}, J.Geom. Anal. {\bf 21} (2011), no. 3, 599-619.

\bibitem{ehz1} P.~Ebenfelt, X. Huang, D. Zaitsev. {\it Rigidity of CR-immersions into spheres}, Comm. Anal. Geom. {\bf 12} (2004), no. 3, 631-670.

\bibitem{ehz2} P.~Ebenfelt, X. Huang and D. Zaitsev. {\it The equivalence problem and rigidity for hypersurfaces embedded into hyperquadrics}, Amer. J. Math. {\bf 127} (2005), no. 1, 169-191.



\bibitem{el} P. Ebenfelt and B. Lamel. {\it Finite jet determination of CR embeddings}, J. Geom. Anal. {\bf 14} (2004), no. 2, 241-265.


\bibitem{es} P. Ebenfelt and R. Shroff. {\it Partial rigidity of CR embeddings of real hypersurfaces into hyperquadrics with small signature difference}, Comm. Anal. Geom. {\bf 23} (2015), no. 1, 159-190.

\bibitem{es1} P. Ebenfelt and D. Son, {\it On the existence of holomorphic embeddings of strictly pseudoconvex algebraic hypersurfaces into spheres}, May, 2012 (arxiv: 1205.1237).

\bibitem{FHX} H. Fang, X. Huang, and M. Xiao, {Volume-preserving mappings between Hermitian symmetric spaces of compact type}, submitted, arxiv: 1602.01900.

\bibitem {Fa} J. J. V. Faran, {\it The nonimbeddability of real hypersurface in spheres}, {\it Proc. Amer. Math. Soc.} {\bf 103}, (1988), no. 3, 902-904.


\bibitem{For} F. Forstneric, {\it Embedding strictly pseudoconvex domains into balls},  Transations of AMS, {\bf 295}   (1986), no. 1, 347-368.

\bibitem{For1} F. Forstneric, {\it Most real-analytic Cauchy-Riemann manifolds are nonalgebraizable}, Manuscripta Math. {\bf 115} (2004), no. 4, 489-494.

\bibitem{HLX} X. Huang, X. Li, and M. Xiao, {\it Non-embeddability into a fixed sphere for a family of compact real
algebraic hypersurfaces}, Int. Math. Res. Not.(IMRN),  DOI: 10.1093/imrn/rnu167.

\bibitem{HZ} X. Huang and D. Zaitsev, {\it Non-embeddable real algebraic hypersurface}, Math. Z., {\bf 275} (2013), no. 3-4, 657-671.

\bibitem{HX} X. Huang and M. Xiao, {\it Chern-Moser-Weyl tensor and embedding into hyperquadrics}, submitted, arxiv: 1606.09145.

\bibitem{HZh} X. Huang and Y. Zhang, {\it Monotonicity for Chern-Moser-Weyl curvature tensor and CR embeddings}, Science in China Series A: Mathematics, {\bf 52} Dec. 2009, no. 12, 2617-2627.

\bibitem{nonanalytic} I.~Kossovskiy and B.~Lamel. {\em On
the Analyticity of CR-diffeomorphisms}. To appear in the American Journal of Math. (AJM).  Available at http://arxiv.org/abs/1408.6711

\bibitem{analytic} I.~Kossovskiy and B.~Lamel. {\it New extension phenomena for solutions of tangential Cauchy-Riemann equations}, Comm. Partial Differential Equations 41 (2016), no. 6, 925–-951.

\bibitem{divergence} I. Kossovskiy and R. Shafikov. {\em Divergent CR-equivalences and meromorphic differential
equations}, 2014.To appear  in J.  Europ. Math. Soc. (JEMS). Available at http://arxiv.org/abs/1309.6799.

\bibitem{nonminimalODE} I. Kossovskiy and R. Shafikov. {\em Analytic
Differential Equations and  Spherical Real Hypersurfaces.} J. Diff. Geom. (JDG),  Vol. 102, No. 1 (2016), pp. 67 - 126.

\bibitem {la1} B. Lamel, {\it A reflection principle for real-analytic submanifolds of complex spaces},  J. Geom. Anal., {\bf 11} (2001), no. 4,  625-631.

\bibitem {la2} B. Lamel, {\it A $C^{\infty}-$regularity for nondegenerate CR mappings},  Monatsh. Math., {\bf 142} (2004), no. 4,  315-326.


\bibitem{merker} J. Merker. {\it Nonrigid spherical real analytic
hypersurfaces in $\CC{2}$}. Complex Var. Elliptic Equ. {\bf 55}
(2010), no. 12, 1155--1182.

\bibitem{ostrowski}  A. Ostrowski. {\em Uber Dirichletsche Reihen und algebraische Differentialgleichungen}, (German) Math. Z. 8 (1920), no. 3-4, 241–298.

\bibitem{segre} B. Segre. {\em Questioni geometriche legate colla teoria delle funzioni di due variabili
complesse}, Rendiconti del Seminario di Matematici di Roma, II,
Ser. 7 (1932), no. 2, 59-107.

\bibitem{sukhov1}  A. Sukhov. {\em Segre varieties and Lie symmetries},
Math. Z. {\bf 238} (2001), no. 3, 483-492.

\bibitem{sukhov2} A. Sukhov. {\em On transformations of analytic
CR-structures}, Izv. Math. {\bf 67} (2003), no. 2, 303--332.

\bibitem{tresse} A. Tresse. {\it D\'etermination des invariants ponctuels de l'\'equation diff\'erentielle
du second ordre $y''=\omega(x,y,y')$}, Hirzel, Leipzig, 1896.



\bibitem  {Wo} K. Wolsson, {\it Linear dependence of a function set of m variables with vanishing generalized Wronskians}, Linear Algebra and its applications, 117 (1989), 73-80.


\bibitem{webster}  S. Webster. {\em On the mappings problem for algebraic
real hyprsurfaces}, Invent. Math., {\bf 43} (1977), no. 1, 53-68.



\bibitem {Z} D. Zaitsev, {\it Obstructions to embeddability into hyperquadrics and explicit examples}, {\it Math. Ann.}, {\bf 342} (2008), no. 3, 695-726.

\end{thebibliography}
\end{document}